\documentclass[11pt, oneside]{amsart}

\usepackage{amsmath,amsfonts,amssymb,amsthm,epsfig,epstopdf,url,array,comment,hyperref,mathrsfs, soul}
\usepackage{hyperref}

\hypersetup{
  colorlinks=true,
  linkcolor=red,
  citecolor=blue,
  urlcolor=black
}

\makeatletter
\renewcommand{\eqref}[1]{%
  \hyperref[#1]{\textup{\tagform@{\ref*{#1}}}}%
}
\makeatother

\usepackage{setspace}
\usepackage{graphicx}
\usepackage{float}
\usepackage{enumerate}

\usepackage{tikz-cd}

\usepackage{fancyhdr}
\usepackage{lipsum} 
\usepackage{geometry} 

\usepackage{fancyhdr}
\usepackage{lipsum} 
\usepackage{etoolbox} 

\usepackage{mathtools} 
\mathtoolsset{showonlyrefs}

\theoremstyle{plain}
\newtheorem{thm}{Theorem}[section]
\newtheorem{lem}[thm]{Lemma}
\newtheorem{cl}[thm]{Claim}
\newtheorem{prop}[thm]{Proposition}
\newtheorem{cor}[thm]{Corollary}

\theoremstyle{definition}
\newtheorem{definition}[thm]{Definition}

\newtheorem{exmpl}[thm]{Example}

\newtheorem{rem}[thm]{Remark}

\newcommand{\D}{\textnormal{D}}
\newcommand{\R}{\mathbb{R}}
\renewcommand*\d{\mathop{}\!\mathrm{d}}
\renewcommand*\D{\mathop{}\!\mathrm{D}}

\newcommand{\<}{\langle}
\renewcommand{\>}{\rangle}

\newcommand{\grad}{\nabla}

\newcommand{\tr}{\textnormal{Tr}}
\newcommand{\hidethis}[1]{}

\newcommand{\supp}{\operatorname{supp}}

\newcommand{\vol}{\textnormal{vol}}

\newcommand{\wasser}{\mathscr{W}_{2}}

\newcommand{\riem}{\operatorname{Riem}}
\newcommand{\ric}{\operatorname{Ric}}
\newcommand{\lap}{\Delta}

\newcommand{\met}{\mathrm g}
\newcommand{\man}{M}
\newcommand{\Tr}{\operatorname{Tr}}

\newcommand{\pheat}{\mathrm P}

\newcommand{\ot}{\theta}
\newcommand{\Id}{\operatorname{Id}}

\newcommand{\PR}{\mathcal P^{\textnormal{ac}}_c(\man)}

\newcommand{\Sec}{\operatorname{sec}}
\newcommand{\geo}{\gamma}

\newcommand{\Jf}{\mathbf{J}}

\newcommand{\ETP}{\mathbf {U}}

\newcommand{\ET}{ \mathbf H}

\newcommand{\Rs}{\mathbf R}
\newcommand{\Matrix}{\mathbf A}
\newcommand{\Ot}{\mathrm F}
\newcommand{\cut}{\operatorname{cut}}
\newcommand{\hess}{\operatorname{Hess}}

\newcommand{\dom}{\operatorname{\mathrm D}}

\newcommand{\dd}{n}

\begin{document}
\title{Sectional curvature and matrix displacement convexity}
\author{Gautam Aishwarya}
\address{Technion -- Israel Institute of Technology, Faculty of Mathematics, Technion City, Haifa 3200003, Israel.}
\email{gautama@technion.ac.il}
\author{Liran Rotem}
\address{Technion -- Israel Institute of Technology, Faculty of Mathematics, Technion City, Haifa 3200003, Israel.}
\email{lrotem@technion.ac.il}
\author{Yair Shenfeld}
\address{Brown University, Division of Applied Mathematics, Providence, RI 02906, USA.}
\email{yair\_shenfeld@brown.edu}
\begin{abstract}
We show that the sectional curvature of a Riemannian manifold is nonnegative if, and only if, the entropy tensor is matrix displacement convex. As an application, under the assumption of nonnegative sectional curvature, we obtain intrinsic dimensional strengthenings of the HWI inequality, the evolution variational inequality, and the corresponding Wasserstein contraction along heat flows. Moreover, we show that the intrinsic dimensional Wasserstein contraction inequality in fact characterizes nonnegative sectional curvature. 
\end{abstract}
\maketitle

\section{Introduction and main results}
\label{sec:intro}

\subsection{Background}
One of the profound consequences of optimal transport theory is the connection between the geometry of a space and the behavior of functionals  over the corresponding Wasserstein space. The classical setting is that of a  smooth complete  Riemannian manifold without boundary $(\man,\met)$, with the induced distance $d$ and volume measure  $\vol$. Associated to $\man$ we have the \emph{Wasserstein space} $(\mathcal{P}_2(\man), \wasser)$, where $\mathcal{P}_2(\man)$ is the set of probability measures  on $\man$ with finite second moment, and $ \wasser$ is the Wasserstein distance between probability measures $\mu_0,\mu_1\in \mathcal{P}_2(\man)$,
\begin{equation} \label{eq: w2def_intro}
    \wasser (\mu_{0}, \mu_{1}) := \left(  \inf_{\pi \in \Pi (\mu_{0}, \mu_{1})} \int_{\man} d^2 (x,y)\d \pi (x,y) \right)^{\frac{1}{2}},
\end{equation}
with $\Pi (\mu_{0}, \mu_{1})$ denoting the collection of all couplings between $\mu_{0}$ and $\mu_{1}$. The functional which is of interest to us is the (Boltzmann) \emph{entropy functional}  given by 
\begin{equation} \label{eq: Hfunctional_intro}
    H ( \mu )  := \begin{cases}
\int_{\man}\ \log \left(\frac{\d \mu}{\d \vol} \right)\d \mu, & \textnormal{ if } \mu \ll \vol , \\
+ \infty, & \textnormal{ otherwise. }  \\
\end{cases}
\end{equation}
The insight that convexity properties of $H$ along Wasserstein geodesics are related to the geometry of $\man$ goes back to formal computations of Otto and Villani  \cite{OttoVillani00}, which showed that if the Ricci curvature of $\man$ is nonnegative, then $t\mapsto H(\mu_t)$ is convex whenever $(\mu_t)_{t\in [0,1]}$ is a Wasserstein geodesic. This result was eventually proven rigorously by Cordero-Erausquin, McCann, and Schmuckenschl\"{a}ger \cite{Cordero-ErasquinMcCannSchmuckenschlager01},  and the converse implication, i.e., that convexity of $H$ along Wasserstein geodesics implies nonnegative Ricci curvature, was established by von-Renesse and Sturm \cite{vonRenesseSturm05}.  Following the   terminology in the seminal work of McCann \cite{McCann97}, we will say that $H$ is \emph{displacement convex} if $[0,1]\ni t\mapsto H(\mu_t)$ is convex for any geodesic $(\mu_t)_{t\in [0,1]}$ in $(\mathcal{P}_{2}(\man), \wasser)$ whenever $\mu_0 \ll \vol$.
\begin{thm}[\cite{OttoVillani00, Cordero-ErasquinMcCannSchmuckenschlager01, vonRenesseSturm05, BolleyGentilGuillinKuwada18}]
\label{thm:Ricci_H_cvx}
Let $(\man, \met)$ be a smooth complete Riemannian manifold without boundary. The following are equivalent.
\begin{enumerate}[(i)]
\item $\man$ has nonnegative Ricci curvature.
\item The entropy functional $H$ is displacement convex.
\item The Wasserstein distance contracts along heat flow. 
\end{enumerate}
\end{thm}
The equivalences stated in Theorem \ref{thm:Ricci_H_cvx} have had a tremendous impact beyond their intrinsic interest,  both in the study of functional inequalities, and in the development of Ricci curvature notions in non-smooth spaces \cite{Villani09}. From the perspective of Riemannian geometry, however, Ricci curvature is only one out of a number of fundamental curvature notions. In this work we are interested in the following question:
\begin{center}
\emph{What is the analogue of Theorem \ref{thm:Ricci_H_cvx} and its ramifications for sectional curvature?}
\end{center}
To this date only a few works have addressed this question.  Kim and Pass \cite{MR4069211} characterized \emph{nonpositive} sectional curvature in terms of the displacement convexity of the \emph{variance} functional. Closer to the question we address here,  Ketterer and Mondino \cite{KettererMondino18} showed that sectional curvature bounds, from below and above, are characterized by the behavior of the appropriate entropy functionals along optimal transport geodesics between \emph{lower-dimensional} probability  measures. On the question of functional inequalities under sectional curvature assumption, even less work has been done.

Our approach relies on the perspective of  \emph{matrix displacement convexity} for the entropy functional. This notion was developed by third-named author in the Euclidean setting \cite{Shenfeld24}, where it was also hinted at the connection between matrix displacement convexity and sectional curvature, but the precise connection remained unclear. Here we show that the entropy functional being  matrix displacement convex is \emph{equivalent} to nonnegative sectional curvature. In turn, this equivalence  allows us to answer the question of what type of functional inequalities one could expect under sectional curvature assumptions: 
\emph{Intrinsic dimensional} functional inequalities.

\subsection{Matrix displacement convexity}
\label{subsec:matrix_disp_cvx_intro}
Let $(\mu_t)_{t\in [0,1]}$ be an optimal transport geodesic in  $(\PR, \wasser)$
where
\begin{equation}
\PR:=\{\mu\in \mathcal{P}(\man): \mu \ll \vol\textnormal{ and $\mu$ is compactly supported}\}.
\end{equation}
It follows from the Brenier--McCann theorem that the optimal transport geodesics are of the form
\begin{equation}
\label{eq:BMthm_intro}
\mu_t= \left[ \exp (t \grad \ot) \right]_{\#} \mu_{0},\qquad\textnormal{for all} \qquad t\in [0,1],
\end{equation} 
where $-\ot:\man \to \R \cup \{ -\infty \}$ is a $\frac{d^{2}}{2}$-concave function, and where $\exp$ is the exponential map on $\man$. The map $\exp (\grad \ot)$ induces a family of Jacobi field matrices $(\Jf_s)_{s\in [0,1]}$ (cf. Section \ref{subsubsec:Jacobi}), and the associated matrices
\begin{equation}
\label{eq:ent_prod_tensor_def_intro}
\ETP_s(x):=\dot{\Jf}_s(x)\Jf_s^{-1}(x),
\end{equation}
where $\dot{\Jf}_s$ stands for the first derivative of $\Jf_s$ in $s$. We define  the (time-dependent) \emph{entropy tensor} associated to $(\mu_0,\mu_1)$ as
\begin{equation}
\label{eq:pointwise_ent_prod_tens_lagrange_intro}
\ET_t^{\mu_0\to\mu_1}(x):=-\int_0^t \ETP_s (x)\d s,\qquad\quad x\in \dom(\mu_0,\mu_1), 
\end{equation}
where $\dom(\mu_0,\mu_1)$ is a set of full $\mu_0$-measure. The idea behind the definition \eqref{eq:pointwise_ent_prod_tens_lagrange_intro} is that the integral of the trace of  $\ET_t^{\mu_0\to\mu_1}$ is essentially the entropy functional $H$. Indeed, by the Jacobi formula,
\begin{equation}
\label{eq:jacobi_formula_intro}
\begin{split}
\frac{\d}{\d t}\log \det (\Jf_t(x))&=\Tr\left[\dot{\Jf}_t(x)\Jf_t(x)^{-1}\right]\overset{\eqref{eq:ent_prod_tensor_def_intro}}{=}\Tr[\ETP_t(x)]\overset{\eqref{eq:pointwise_ent_prod_tens_lagrange_intro} }{=}-\frac{\d}{\d t}\Tr[\ET_t^{\mu_0\to\mu_1}(x)],
\end{split}
\end{equation}
so using the change of variables formula
\begin{equation}
\label{eq:cov_entropy_log_intro}
H(\mu_0)=H(\mu_t)+\int_{\man}\log \det (\Jf_t(x))\d\mu_0(x),
\end{equation}
together with \eqref{eq:jacobi_formula_intro}, gives
\begin{equation}
\label{eq:tensor_ent_trace_intro}
H(\mu_t)=H(\mu_0)+\int_{\man}\Tr[\ET_t^{\mu_0\to\mu_1}(x)]\d\mu_0(x).
\end{equation}
(We refer to Lemma \ref{lem:tr_ent_mat_is_ent} for more details.)
Hence, the concept of the entropy tensor allows us to ``decompose" the contribution to the entropy $H(\mu_t)$ coming from different directions in space, which will lead to our intrinsic dimensional inequalities in Section \ref{subsubsec:intrinsic_dim_funq_inq_intro}. Let us denote the collection of  entropy tensors as 
\begin{equation}
\label{eq:H_gen_def_intro}
\ET:=\{(\ET_t^{\mu_0\to\mu_1}(x))_{t\in [0,1],x\in\dom(\mu_0,\mu_1)}\}_{\mu_0,\mu_1\in \PR}.
\end{equation}
\begin{definition}[Matrix displacement convexity]
\label{def:matrix_disp_cvx_intro}
The entropy tensor $\ET$  is \emph{matrix displacement convex} if for every $\mu_0,\mu_1\in \PR$, we have, for all $t\in [0,1]$ and $x\in \dom(\mu_0,\mu_1)$, 
\begin{equation}
\label{eq:matrix_disp_cvx_def__2nd_der_intro}
\ddot{\ET}_{t}^{\mu_0\to\mu_1}(x) \succeq 0,
\end{equation}
where $\succeq$ is the Loewner order, and where $\ddot{\ET}_{t}^{\mu_0\to\mu_1}(x)$ stands for the second derivative of $\ET_{t}^{\mu_0\to\mu_1}(x)$ in $t$.
\end{definition}

We will elaborate on the definition of matrix displacement convexity in the following section, but to give a quick intuition, note that if $\ET$ is matrix displacement convex then
\begin{equation}
\label{eq:_matx_cx_implies_cvx_intro}
\ddot{H}(\mu_t)\overset{\eqref{eq:tensor_ent_trace_intro}}{=}\int_{\man}\Tr\left[\ddot{\ET}_t^{\mu_0\to\mu_1}(x)\right]\d\mu_0(x)\ge 0.
\end{equation}
In words, matrix displacement convexity of the entropy tensor $\ET$ implies the displacement convexity of the entropy functional $H$. Equation \eqref{eq:_matx_cx_implies_cvx_intro} already hints at the connection between matrix displacement convexity and sectional curvature. As we saw above,
\begin{equation}
\label{eq:Ricci_H_intro}
\Tr[\riem]=\ric\ge 0\qquad \overset{\textnormal{Theorem \ref{thm:Ricci_H_cvx} }}{\Longleftrightarrow}\qquad\ddot{H}(\mu_t)\ge 0,
\end{equation}
and below will show that 
\begin{equation}
\label{eq:Riem_H_intro}
\riem\succeq 0\qquad \overset{\textnormal{Theorem \ref{thm: mainsecequiv_intro} }}{\Longleftrightarrow}\qquad  \ddot{\ET}_{t}^{\mu_0\to\mu_1}(x)\succeq 0.
\end{equation}
Thus, the relation \eqref{eq:tensor_ent_trace_intro} between the integral of the entropy tensors and the entropy functional mirrors the relation between the Riemann curvature tensor and the Ricci curvature tensor.

\begin{rem}
\label{rem:euler_formulation}
In the Eulerian picture of optimal transport the Wasserstein geodesic $(\mu_t)_{t\in [0,1]}$ evolves according to a continuity equation 
\begin{equation}
\label{eq:cont_eq_intro}
\partial_t\mu_t(x)+\grad \cdot(\mu_t(x)\grad\ot_t(x))=0,\qquad  \textnormal{for}\quad t\in (0,1)\quad\textnormal{and} \quad x\in\man,
\end{equation}
where $\ot_t(x)$ satisfies the Hamilton--Jacobi equation, 
\begin{equation}
\label{eq:HJ_intro}
\partial_t\ot_t(x)+\frac{|\nabla\ot_t(x)|^2}{2}=0,
\end{equation}
with $\ot_0=\ot$ the optimal transport potential, $\exp(\grad\ot)_{\sharp}\mu_0=\mu_1$. The usual correspondence between the Eulerian and Lagrangian pictures \cite[Eq. (14.25)]{Villani09} shows that, up to identification of tangent spaces via parallel transport,
\begin{equation}
\label{eq:pointwise_ent_Euler_lagrange}
\ETP_s(x)=\grad^2\ot_s(\exp_x (s \grad \ot(x))),
\end{equation} 
so we may write
\begin{equation}
\label{eq:pointwise_ent_prod_tens_euler_intro}
\ET_t^{\mu_0\to\mu_1}(x):=-\int_0^t\grad^2\ot_s(\exp_x (s \grad \ot(x)))\d s. 
\end{equation}
However, true to the slogan “Think Eulerian, prove Lagrangian" \cite[p. 444]{Villani09}, while the Eulerian formulation is cleaner and more conceptual, it is less useful than the Lagrangian formulation for the purpose of rigorous proofs. Thus, we will stick from now on with the Lagrangian perspective. 
\end{rem}

\subsection{Sectional curvature and matrix displacement convexity}
\label{subsec:sec_mat_funinq_intro}
The first main result of our work is the following analogue\footnote{Compared to Theorem \ref{thm:Ricci_H_cvx} we make a compactness assumption on the support of the measures in Theorem \ref{thm: mainsecequiv_intro}; see Remark \ref{rem:compact_assumption}.} of Theorem \ref{thm:Ricci_H_cvx}. 

\begin{thm}
\label{thm: mainsecequiv_intro}
Let $(\man, \met)$ be a smooth complete Riemannian manifold without boundary. Then, the following are equivalent.
\begin{enumerate}[(i)]
\item \label{enum:sec} $\man$ has nonnegative sectional curvature.
\item \label{enum:cvx} $\ET$ is matrix displacement convex. 
\item \label{enum:strong_sec} For any $\mu_0,\mu_1\in \PR$, we have, for $t\in [0,1]$ and $x\in \dom(\mu_0,\mu_1)$, 
\begin{equation}
\label{eq:entropy_Tensor_diff_inq_strong_intro}
 \ddot{\ET}_{t}^{\mu_0\to\mu_1}(x)\succeq [\dot{\ET}_{t}^{\mu_0\to\mu_1}(x)]^2.
\end{equation}
\end{enumerate}
\end{thm}
The equivalence \eqref{enum:sec}$\Leftrightarrow$\eqref{enum:cvx} is the tensorial analogue of  Theorem \ref{thm:Ricci_H_cvx}. The stronger convexity of $\ET$, namely \eqref{eq:entropy_Tensor_diff_inq_strong_intro}, is the tensorial analogue of the following classical result of Erbar, Kuwada, and Sturm \cite{ErbarKuwadaSturm15}.
\begin{thm}[\cite{ErbarKuwadaSturm15}]
\label{thm:EKS_intro}
Let $(\man, \met)$ be a smooth complete $n$-dimensional Riemannian manifold without boundary. Then, the following are equivalent.
\begin{enumerate}[(i)]
\item $\man$ has nonnegative Ricci curvature.
\item $\ddot{H}(\mu_t)\ge \frac{[\dot{H}(\mu_t)]^2}{n}$.
\end{enumerate}
\end{thm}
The improvement of Theorem \ref{thm:EKS_intro} over Theorem \ref{thm:Ricci_H_cvx}, i.e., upgrading $\ddot{H}(\mu_t)\ge 0$ to  $\ddot{H}(\mu_t)\ge \frac{[\dot{H}(\mu_t)]^2}{n}$, is due to the explicit assumption that $\man$ is of dimension $n$.  Our inequality $\ddot{\ET}_{t}^{\mu_0\to\mu_1}(x)\succeq [\dot{\ET}_{t}^{\mu_0\to\mu_1}(x)]^2$ in \eqref{eq:entropy_Tensor_diff_inq_strong_intro} is the tensorial analogue of $\ddot{H}(\mu_t)\ge \frac{[\dot{H}(\mu_t)]^2}{n}$. The dimensional information in Theorem \ref{thm: mainsecequiv_intro} is implicit in the definition of matrix displacement convexity, in the sense that  \eqref{eq:matrix_disp_cvx_def__2nd_der_intro} is a statement about $n$ eigenvalues. However, the fact that the dimension $n$ does not appear explicitly in the definition of matrix displacement convexity will turn out to have advantageous consequences for functional inequalities. Indeed, while  Theorem \ref{thm: mainsecequiv_intro} on its own is not very surprising---the connection between Jacobi fields and curvature is classical---the perspective of viewing entropy as the trace of the entropy tensor (defined via Jacobi fields) is what leads to the new intrinsic dimensional functional inequalities in the upcoming section. 

\subsection{Intrinsic dimensional functional inequalities} 
\label{subsubsec:intrinsic_dim_funq_inq_intro}
One of the main applications of Theorem \ref{thm:Ricci_H_cvx} is the ability to prove functional inequalities for Riemannian manifolds with nonnegative Ricci curvature \cite{OttoVillani00}. Analogously, Theorem \ref{thm:EKS_intro} allows to prove \emph{dimensional} functional inequalities, where the inequality improves using the dimensional information. In the Euclidean setting, it was shown in \cite{Shenfeld24} how matrix displacement convexity leads to \emph{intrinsic dimensional} functional inequalities (see also the work of Eskenazis and the third-named author \cite{EskenazisShenfeld24}). These type of inequalities are stronger than their dimensional counterparts, capturing different behaviors along different dimensions. Let us demonstrate this line of thinking with the classical evolution variational inequalities (EVI) and Wasserstein contraction along heat flows. Let $(\pheat_t)_{t\ge 0}$ be the heat semigroup on $\man$ generated by  the Laplace--Beltrami operator. The \emph{evolution variational inequality} \cite[Theorem 4.2]{DaneriSavare08} states that, when the Ricci curvature of $\man$ is nonnegative, we have\footnote{To simplify the exposition we use derivatives rather than super-derivatives. We will address this issue in Section \ref{sec:proofApplication}.}
\begin{equation}
\label{eq:EVI_intro}
\frac{\d}{\d \tau}\wasser^{2} (\pheat_{\tau}\mu_{0}, \mu_{1}) \leq 2\,[H(\mu_1)-H(\pheat_{\tau}\mu_{0})].
\end{equation}
The inequality \eqref{eq:EVI_intro} follows from Theorem \ref{thm:Ricci_H_cvx} using the convexity of $H$ along Wasserstein geodesics. Parallel to the relation between Theorem \ref{thm:Ricci_H_cvx} and Theorem \ref{thm:EKS_intro}, the dimensional information on $\man$ can be used to upgrade \eqref{eq:EVI_intro} into the \emph{dimensional evolution variational inequality} \cite[Theorem 2]{ErbarKuwadaSturm15},
\begin{equation}
\label{eq:dim_EVI_intro}
\frac{\d}{\d \tau}\wasser^{2} (\pheat_{\tau}\mu_{0}, \mu_{1}) \leq 2n \left(1-e^{-\frac{1}{n}[H(\mu_1)-H(\pheat_{\tau}\mu_{0})]}\right).
\end{equation}
The inequality \eqref{eq:dim_EVI_intro} improves on \eqref{eq:EVI_intro} since $e^{-x}\ge 1-x$. Our next result\footnote{In the following results we write $h(\Matrix):=Uh(\Lambda) U^{-1}$ for a symmetric matrix $\Matrix$ with eigen-decomposition $\Matrix=U\Lambda U^{-1}$, and a function $h:\R\to\R$,  where $h(\Lambda) $ is the diagonal matrix obtained by applying $h$ pointwise to the eigenvalues in $\Lambda$.} provides an \emph{intrinsic dimensional} improvement on \eqref{eq:dim_EVI_intro}, and hence also on \eqref{eq:EVI_intro}.

\begin{thm}[Intrinsic dimensional evolution variational  inequality]\label{thm:intrinsic_EVI_intro}
Let $(\man, \met)$ be a smooth compact $n$-dimensional Riemannian manifold without boundary, and let $(\pheat_t)$ be the  heat semigroup on $\man$. Suppose that the sectional curvature of $\man$ is nonnegative. Then, for all $\mu_0,\mu_1\in \PR$ with $H(\mu_1)<\infty$, given $T>0$, we have, for almost all $\tau\in (0,T)$,
\begin{equation}
\label{eq:intrinsic_EVI_intro}
\frac{\d}{\d \tau}\wasser^{2} (\pheat_{\tau}\mu_{0}, \mu_{1}) \leq 2n\int_{\man}\left(1 - \frac{\Tr[e^{-\ET_1^{\pheat_{\tau}\mu_0\to\mu_1}(x)}]}{n}\right)\d \pheat_{\tau}\mu_0(x).
\end{equation}
\end{thm}
The inequality \eqref{eq:intrinsic_EVI_intro} improves on \eqref{eq:dim_EVI_intro} by a repeated application of  Jensen's inequality,
\begin{equation}
\begin{split}
\label{eq:jenen_twice_intro}
& \frac{1}{n}\int_{\man}\Tr[e^{-\ET_1^{\pheat_{\tau}\mu_0\to\mu_1}(x)}]\d \pheat_{\tau}\mu_0(x)\ge e^{-\frac{1}{n}\int_{\man}\Tr[\ET_1^{\pheat_{\tau}\mu_0\to\mu_1}(x)]\d \pheat_{\tau}\mu_0(x)}\overset{\eqref{eq:tensor_ent_trace_intro}}{=}e^{-\frac{1}{n}[H(\mu_1)-H(\pheat_{\tau}\mu_0)]}.
 \end{split}
\end{equation}
Let us state a corollary of Theorem \ref{thm:intrinsic_EVI_intro}, which will also clarify the sense in which these inequalities are ``intrinsic dimensional".  As is classical, evolution variational inequalities lead to contraction in Wasserstein distance along the heat flow. The diffusion of the heat flow blurs the distinction between $\pheat_T\mu_{0}$ and $\pheat_T\mu_1$ for $T>0$, so we expect the Wasserstein distance between them to not increase. Indeed,  \cite[Corollary 1]{vonRenesseSturm05} established the following contraction result, which can also be proved by the evolution variational inequality \eqref{eq:EVI_intro},
\begin{equation}
\label{eq:wass_contract_intro}
 \wasser^{2} (\pheat_T\mu_{0}, \pheat_T\mu_{1})\le  \wasser^{2} (\mu_{0}, \mu_{1}).
\end{equation}
Incorporating the dimensional information on $\man$, it was shown in \cite[Theorem 2.1]{BolleyGentilGuillinKuwada18}  that the dimensional evolution variational inequality \eqref{eq:dim_EVI_intro} leads to the dimensional improvement on \eqref{eq:wass_contract_intro}, 
\begin{equation}
\label{eq:dim_wass_contract_intro}
 \wasser^{2} (\pheat_T\mu_{0}, \pheat_T\mu_{1})\le  \wasser^{2} (\mu_{0}, \mu_{1})-8n\int_0^T \sinh^2\left(\frac{H(\pheat_{\tau}\mu_1)-H(\pheat_{\tau}\mu_{0})}{2n}\right)\d \tau. 
\end{equation}
Analogously, we will deduce from the intrinsic dimensional  evolution variational inequality \eqref{eq:intrinsic_EVI_intro} the following improvement of Wasserstein contraction along heat flows.

\begin{cor}[Intrinsic dimensional Wasserstein contraction along heat flow]
\label{cor:W_contract_intro}
Let $(\man, \met)$ be a smooth compact $n$-dimensional Riemannian manifold without boundary, and let $(\pheat_t)$ be the  heat semigroup on $\man$. Suppose that the sectional curvature of $\man$ is nonnegative. Then, for all $\mu_0,\mu_1\in \PR$ with $H(\mu_0),H(\mu_1)<\infty$, we have, for any $T>0$, 
\begin{equation}
\label{eq:intrinsic_wass_contract_intro}
 \wasser^{2} (\pheat_T\mu_{0}, \pheat_T\mu_{1})\le  \wasser^{2} (\mu_{0}, \mu_{1})-8\int_0^T\int_{\man} \Tr\left[\sinh^2\left(\frac{\ET_1^{\pheat_{\tau}\mu_0\to\pheat_{\tau}\mu_1}(x)}{2}\right)\right]\d \pheat_{\tau}\mu_0(x)\d \tau.
\end{equation}
\end{cor}
Again,  Jensen's inequality shows that \eqref{eq:intrinsic_wass_contract_intro} improves on \eqref{eq:dim_wass_contract_intro}, and hence also on \eqref{eq:wass_contract_intro}. The following example demonstrates how the intrinsic dimension can be read off from \eqref{eq:intrinsic_wass_contract_intro}. 

\begin{exmpl}[Product manifolds] 
\label{ex:prod}
Suppose
\begin{equation}
(\man,\met)=(\man_k,\met^k)\times (\man_{n-k},\met^{n-k})
\end{equation}
is a product manifold of two manifolds with dimensions $k$ and $n-k$, respectively (the upper-scripts in $\met^k$ and $\met^{n-k}$ signify the dimension of the underlying manifolds). Let $\mu_0,\mu_1\in\PR$ be a product probability measure of the form
\begin{equation}
\mu_0=\rho_0\otimes \nu\qquad\textnormal{and}\qquad \mu_1=\rho_1\otimes \nu,
\end{equation}
for some $\rho_0,\rho_1\in \mathcal{P}_{\textnormal{c}}^{\textnormal{ac}}(\man_k)$ and  $\nu\in \mathcal{P}_{\textnormal{c}}^{\textnormal{ac}}(\man_{n-k})$. Since $\mu_0$ and $\mu_1$ are identical in $n-k$ coordinates, the intrinsic dimension of the Wasserstein geodesic connecting them should be $k$ rather than $n$. Dimensional inequalities such as  \eqref{eq:dim_wass_contract_intro} cannot detect this structure since they only exploit information about the \emph{ambient} space. In contrast, let us see the form that \eqref{eq:intrinsic_wass_contract_intro} takes in this example. The potential $\ot$ of the optimal transport map between $\mu_0$ and $\mu_1$ will be of the form
\begin{equation}
\ot(x)=\ot_k(x^k),\quad \textnormal{for}\quad x=(x^k,x^{n-k})\in\man,\quad x^k\in\man_k, \quad x^{n-k}\in\man_{n-k},
\end{equation}
where $ -\ot_k:\man_{k}\to \R\cup\{-\infty\}$  is $\frac{d_k^2}{2}$-concave function (with $d_k$ the distance on $\man_k$ induced by $\met^k$). A simple calculation shows that
\begin{equation}
\label{eq:U_prod}
\ETP_s=\begin{bmatrix}
\ETP_s^k & 0_{k\times (n-k)}\\
0_{(n-k)\times k} & 0_{(n-k)\times (n-k)}
\end{bmatrix},
\end{equation}
where $\ETP_s^k$ is the $k\times k$ matrix defined by \eqref{eq:ent_prod_tensor_def_intro}, with the Jacobi fields associated to $\grad\ot^k$ on $\man_k$. Hence, by the definition
 \eqref{eq:pointwise_ent_prod_tens_lagrange_intro},  it follows that 
 \begin{align}
\label{eq:ent_tensor_prod_intro}
\tr\left[\sinh^2\left(\frac{\ET_1^{\pheat_{\tau}\mu_0\to\pheat_{\tau}\mu_1}(x)}{2}\right)\right]=\tr\left[\sinh^2\left(\frac{\ET_1^{\pheat_{\tau}^k\rho_0\to\pheat_{\tau}^k\rho_1}(x)}{2}\right)\right],
\end{align}
where $(\pheat_t^k)$ stands for the heat semigroup on $(\man_k,\met^k)$. Consequently, \eqref{eq:intrinsic_wass_contract_intro} now reads 
\begin{align}
\begin{split}
\label{eq:intrinsic_wass_contract_prod_intro}
 \wasser^{2} (\pheat_T\mu_{0}, \pheat_T\mu_{1})&\le  \wasser^{2} (\mu_{0}, \mu_{1})-8k\int_0^T\int_{\man} \frac{1}{k}\Tr\left[\sinh^2\left(\frac{\ET_1^{\pheat_{\tau}^k\rho_0\to\pheat_{\tau}^k\rho_1}(x)}{2}\right)\right]\d \pheat_{\tau}\mu_0(x)\d \tau\\
 &\le \wasser^{2} (\mu_{0}, \mu_{1})-8k\int_0^T \sinh^2\left(\frac{H(\pheat_{\tau}\mu_1)-H(\pheat_{\tau}\mu_{0})}{2k}\right)\d \tau,
\end{split} 
\end{align}
where the last inequality is by Jensen's inequality, and the additivity of $H$ with respect to product measures. Thus, the intrinsic dimensional inequality \eqref{eq:intrinsic_wass_contract_intro} facilitates the replacement of the ambient dimension $n$ in the dimensional inequality \eqref{eq:dim_wass_contract_intro} by the intrinsic dimension $k$. 
\end{exmpl}
Example \ref{ex:prod} demonstrates how our results capture the intrinsic dimension in settings where the separation between intrinsic and ambient dimensions is clear due to a product structure. However, our results are much more robust. Namely, they can capture settings which are ``approximately of product structure", in the sense that some of the eigenvalues of $\ET$ are small, which will translate to the last term  in \eqref{eq:intrinsic_wass_contract_intro} being of the correct order. 
\bigskip

So far we established in Theorem \ref{thm: mainsecequiv_intro} the analogues of the equivalence (i)$\Leftrightarrow$(ii) of Theorem \ref{thm:Ricci_H_cvx}. We now turn to the analogue of part (iii) in Theorem \ref{thm:Ricci_H_cvx}. Specifically, 
we show that the Wasserstein contraction inequality \eqref{eq:intrinsic_wass_contract_intro} is strong enough to characterize nonnegative sectional curvature. 
\begin{thm}[Intrinsic dimensional Wasserstein contraction characterizes nonnegative sectional curvature]
\label{thm:W_contract_char_sec_intro}
Let $(\man, \met)$ be a smooth compact $n$-dimensional Riemannian manifold without boundary. Suppose the intrinsic dimensional Wasserstein contraction along heat flow \eqref{eq:intrinsic_wass_contract_intro} holds  for all $\mu_0,\mu_1\in \PR$ and $T>0$. Then, $\man$ has nonnegative sectional curvature. 
\end{thm}
\begin{rem}
Since the intrinsic dimensional Wasserstein contraction  \eqref{eq:intrinsic_wass_contract_intro} follows from the EVI inequality \eqref{eq:intrinsic_EVI_intro}, Theorem \ref{thm:W_contract_char_sec_intro} shows that the validity of EVI implies nonnegative sectional curvature. One could also prove this result directly similarly to how Theorem \ref{thm:W_contract_char_sec_intro} is proven. 
\end{rem}
We conclude the section by establishing another intrinsic dimensional inequality: the \emph{intrinsic dimensional HWI inequality}. 
Here "H" stands for the entropy, "W" stands for the Wasserstein distance, and "I" stands for the Fisher information. Recall that the Fisher information is defined as    
\begin{equation}
\label{eq:Fisher_information_def_intro}
I(\mu):=\int_{\man}\left|\grad\log\left(\frac{\d \mu}{\d \vol}\right)\right|^2\d\mu
\end{equation}
for every $\mu\in\PR$ with strictly positive smooth density with respect to $\vol$. 

The original HWI inequality, due to Otto and Villani \cite{OttoVillani00}, states that if $M$ has nonnegative Ricci curvature and $\vol(M)=1$ then for every such $\mu\in\PR$ we have 
\begin{equation}
\label{eq:HWI_intro}
H(\mu)\leq
\wasser(\mu,\vol)\sqrt{I(\mu)}.
\end{equation}
Erbar, Kuwada and Sturm \cite{ErbarKuwadaSturm15} improved (\ref{eq:HWI_intro}) to a \emph{dimensional} HWI inequality:
Under the same assumption of nonnegative Ricci curvature, and the additional assumption that the manifold is of dimension $n$, we have
\begin{equation}
\label{eq:dimensional_HWI_intro}
n\left(e^{\frac{1}{n}H(\mu)}-1\right)
\leq
\wasser(\mu,\vol)\sqrt{I(\mu)}.
\end{equation}
Our next result is an \emph{intrinsic dimensional} improvement of (\ref{eq:dimensional_HWI_intro}). 
\begin{thm}[Intrinsic dimensional HWI inequality]
\label{thm:HWI_intro}
\label{thm:intrinsic_HWI_intro}
Let $(\man,\met)$ be a smooth compact $n$-dimensional Riemannian manifold without boundary. Suppose that the sectional curvature of $\man$ is nonnegative and that $\vol(\man)=1$. Then, for any $\mu\in\PR$ which is absolutely continuous with respect to $\vol$ with strictly positive smooth density, we have
\begin{equation}
\label{eq:intrinsicdimHWI_Intro}
\int_{\man}
\left(
\Tr\left[e^{-\ET_1^{\mu\to\vol}(x)}\right]-n
\right)
\d\mu(x)
\leq
\wasser(\mu,\vol)\sqrt{I(\mu)}.
\end{equation}
\end{thm}

\begin{rem}
\label{rem:remove_compact}
We state several results, such as Theorem \ref{thm:intrinsic_EVI_intro} and Corollary \ref{cor:W_contract_intro}, under compactness assumptions on $\man$. This assumption is made for technical convenience and can potentially be relaxed; see  \cite[Theorem 23.19]{Villani09}.
\end{rem}

\subsection*{Organization of paper} 
In Section \ref{sec:proofMain} we prove Theorem \ref{thm: mainsecequiv_intro} on   the characterization of nonnegative sectional curvature in terms of matrix displacement convexity of the entropy tensor, as well as tensorial versions of Bochner inequalities. In Section \ref{sec:proofApplication} we prove the intrinsic dimensional evolution variational inequalities and Wasserstein contraction along heat flows (Theorem \ref{thm:intrinsic_EVI_intro}, Corollary \ref{cor:W_contract_intro} and Theorem \ref{thm:intrinsic_HWI_intro}). In Section \ref{sec:converse} we prove Theorem \ref{thm:W_contract_char_sec_intro}.

\subsection*{Acknowledgments} 

We thank Emanuel Milman for helpful comments on this manuscript, and Jordan Serres for insightful discussions on this topic. We thank the referee for their very useful remarks which encouraged us to discover  Theorem \ref{thm:W_contract_char_sec_intro} and Theorem \ref{thm:intrinsic_HWI_intro}. 

This material is based upon work supported by the National Science Foundation under Award Numbers DMS-2331920, DMS-2508545, and NSF-DMS 2154402. The second named author is also supported by Israel Science Foundation grant no. 2574/24	and NSF-BSF grant 2022707. The first-named author gratefully acknowledges support from the Department of Mathematics at Michigan State University, where he was based during part of this work.

\subsection*{AI disclosure}
We learned of the application of Lemma \ref{lem:exp_mat} to our setting from ChatGPT o3, and benefited from discussions with ChatGPT 5.5 in proving Theorem \ref{thm:W_contract_char_sec_intro}. ChatGPT 5.6 pointed out that the choice $s=\tau$ in the proof of Corollary \ref{cor:W_contract} required additional justification and assisted us in clarifying that step.

\section{Sectional curvature and matrix displacement convexity}
\label{sec:proofMain}
In this section we prove our first main result, Theorem \ref{thm:main}, which characterizes nonnegative sectional curvature in terms of matrix displacement convexity of the entropy tensor, as well as tensorial versions of Bochner inequalities. We begin with Section \ref{subsec:prelim_riemann} and  Section \ref{subsec:prelim_ot}  which, for completeness,  provide the standard necessary preliminaries from Riemannian geometry and optimal transport. We then prove Theorem \ref{thm:main} in Section \ref{subsec:main_result}. 
\subsection{Preliminaries of Riemannian geometry}
\label{subsec:prelim_riemann}
In this section we present the classical notions and formulas from Riemannian geometry that will be needed for our work; we refer to \cite{GallotHulinLafontaine04, Petersen16} and \cite[\S 14]{Villani09} for more details. In Section \ref{subsub:curvature} we review the notion of curvature in Riemannian manifolds, and in Section \ref{subsubsec:Jacobi} we review Jacobi fields and the Jacobi equation. Using the Jacobi equation we derive in Section \ref{subsub:Bochner} a tensorial version of the Bochner formula.

\subsubsection{Curvature}
\label{subsub:curvature}
 Let $(\man,\met)$ be  a smooth complete $n$-dimensional Riemannian manifold without boundary. We denote by $\grad$ the Levi--Civita connection on $\man$, and recall that  for smooth vector fields $X, Y$ on $\man$ we have the  \emph{Riemann curvature tensor},
\begin{equation}
\label{eq:riem_def}
    \riem (X,Y) := \grad_{Y} \grad_{X} - \grad_{X} \grad_{Y} + \grad_{[X,Y]}.
\end{equation}
The relation between the \emph{sectional curvature}  and the Riemann curvature tensor goes through the equation, for any $x\in \man$ and a 2-dimensional plane $P\subseteq T_x\man$ in the tangent space $T_x\man$,
\begin{equation}
\label{eq:riem_sectional}
\Sec_x(P) = \frac{\met_x( \riem_x(u,v)u,v)}{|u|^2|v|^2-\met_x(u,v)^2},
\end{equation}
where $u,v\in T_x\man$ are any two vectors which span $P$. The \emph{Ricci curvature tensor} is the trace of the Riemann curvature tensor, 
\begin{equation}
\label{eq:ric_def}
    \ric (X,Y) :=\sum_{i=1}^n \met(e^i,\riem (X,e^i)Y),
\end{equation}
where $\{e^i\}_{i=1}^n$ is any orthonormal frame (the right-hand side of \eqref{eq:ric_def} is independent of the  choice of the frame). In terms of sectional curvature we have, for a given  $x\in \man$ and a unit vector $e\in T_x\man$,
\begin{equation}
\label{eq:ric_sec}
\ric_x(e,e)=\sum_{j=2}^n\sec_x(P_j),
\end{equation}
where, for $j=2,\ldots, n$, $P_j\subseteq T_x\man$ is the 2-dimensional plane spanned by $\{e,e^j\}$, with $e^2,\ldots,e^n\in T_x\man$ such that $\{e,e^2,\ldots,e^n\}$ forms an orthonormal basis of $T_x\man$.

\subsubsection{Jacobi fields} 
\label{subsubsec:Jacobi}
For a curve $(\geo_s)$ in $\man$ we denote by $\left(\frac{\d}{\d s}\geo_s\right)$ its velocity field, $\frac{\d}{\d s}\geo_s\in T_{\geo_s}\man$. A curve $(\geo_s)$ is a \emph{geodesic} if its acceleration is zero, 
\begin{equation}
\label{eq:geo_def}
\frac{\D}{\D s}\frac{\d}{\d s}\geo_s=0\qquad\textnormal{for all $s$},
\end{equation}
where $\frac{\D}{\D s}$ is the covariant derivative along $(\geo_s)$. Let $(\geo_s)$ be a geodesic, and let $r\mapsto \geo_s(r)$ be a geodesic perturbation of $\geo$, i.e.,  $\geo_s(0)=\geo_s$ for all $s$, and $s\mapsto \geo_s(r)$ is a geodesic for each fixed $r$.  By \eqref{eq:geo_def} we have $\frac{\D}{\D s}\frac{\d}{\d s}\geo_s(r)=0$ for all $s$ and $r$, so, taking the  covariant derivative  $\frac{\D}{\D r}$ of this equation along the curve $r\mapsto \geo_s(r)$, we get 
\begin{equation}
\label{eq:3der}
\frac{\D}{\D r}\frac{\D}{\D s}\frac{\d}{\d s}\geo_s=0\qquad\textnormal{for all $s$ and $r$}. 
\end{equation}
Exchanging the derivatives in \eqref{eq:3der} yields 
\begin{equation}
\label{eq:3der_exchange}
\riem\left(\frac{\d}{\d s}\geo_s(r),\frac{\d}{\d r}\geo_s(r)\right)\frac{\d}{\d s}\geo_s(r)+\frac{\D^2}{\D s^2}\frac{\d}{\d r}\geo_s(r)=0.
\end{equation}
Hence, defining the \emph{Jacobi field} 
\begin{equation}
\label{eq:Jacobi_def}
J_s:=\frac{\d}{\d r}\geo_s(r)\bigg|_{r=0},
\end{equation}
we get the \emph{Jacobi equation}
\begin{equation}
\label{eq:Jacobi_eq_1}
\frac{\D^2}{\D s^2} J_s+\riem\left(\frac{\d}{\d s}\geo_s, J_s\right)\frac{\d}{\d s}\geo_s=0.
\end{equation}
Conversely, being a second-order equation, equation \eqref{eq:Jacobi_eq_1} will have a unique solution as soon as we specify the geodesic $(\geo_s)$, $J_0$, and $\frac{\D}{\D s}J_s|_{s=0}$. This solution will correspond to some geodesic perturbation $(r,s)\mapsto\geo_s(r)$.  

Let us consider a family of Jacobi fields  $(J_s^i)$, for $i=1,\ldots,n$, along a given geodesic $(\geo_s)$. Set $\{e^i_0\}_{i=1}^n$ to be an orthonormal basis of $T_{\geo_0}\man$ such that $e^1_0\propto \frac{\d}{\d s}\geo_s |_{s=0}$, and let $\{e^i_s\}_{i=1}^n$ be the parallel transport of $\{e^i_0\}_{i=1}^n$ along $(\geo_s)$, i.e., $\frac{\D}{\D s}e^i_s=0$ for all $s$ and $i$. It will be convenient to describe the evolution of this family of Jacobi fields using moving frames. Let $(\Jf_s)$ be the $n\times  n$ matrix whose columns are the coordinates of $(J_s^1,\ldots,J_s^n)$ expressed in the basis  $\{e^i_s\}_{i=1}^n$. Then $(\Jf_s)$ solves the equation, in the space of $n\times n$ matrices,
\begin{equation}
\label{eq:jacobi_eq_matrix}
\ddot{\Jf}_s+\Rs_s\Jf_s=0,
\end{equation}
where $\dot{\Jf}_s, \ddot{\Jf}_s$ stand for the first and second derivatives, respectively, in $s$, and where $\Rs_s$ is an $n\times n$ matrix whose entries $\Rs_s=\{\Rs_s^{ij}\}_{i,j=1}^n$ are given by 
\begin{equation}
\label{eqRs_def}
\Rs_s^{ij}:=\met_{\geo_s}\left(\riem_{\geo_s}\left(\frac{\d}{\d s}\geo_s, e^i_s\right)\frac{\d}{\d s}\geo_s,e^j_s\right).
\end{equation}
When $\Jf_s$ is invertible for all $s$ we set 
\begin{equation}
\label{eq:Us_def}
\ETP_s:=\dot{\Jf}_s\Jf_s^{-1}.
\end{equation}
Differentiating \eqref{eq:Us_def}, and using \eqref{eq:jacobi_eq_matrix}, we find that $(\ETP_s)$ satisfies the equation 
\begin{equation}
\label{eq:Us_jacobi}
\dot{\ETP}_s+\ETP_s^2+\Rs_s=0. 
\end{equation}
As we saw in Section \ref{subsec:matrix_disp_cvx_intro}, the  entropy tensor $\ET_t^{\mu_0\to\mu_1}(x)$ is defined in terms of $(\ETP_s)_{s\in [0,t]}$, and the equation \eqref{eq:Us_jacobi} provides the link between the Riemann curvature tensor and the entropy tensor. This link is the analogue of the following classical connection between entropy and Ricci curvature. Taking the trace in \eqref{eq:Us_jacobi} we find that
\begin{equation}
\label{eq:U_ode_tr}
\frac{\d}{\d s}\Tr[\ETP_s]+\Tr[\ETP_s^2]+\ric\left(\frac{\d}{\d s}\geo_s,\frac{\d}{\d s}\geo_s\right)=0,
\end{equation}
which almost provides an equation for $\Tr[\ETP_s]$, except for the term $\Tr[\ETP_s^2]$. Using Jensen's inequality, $\Tr[\ETP_s^2]\ge \frac{\Tr^2[\ETP_s]}{n}$, we find the differential inequality for $s\mapsto \Tr[\ETP_s]$,
\begin{equation}
\label{eq:U_ode_inq_tr}
\frac{\d}{\d s}\Tr[\ETP_s]+ \frac{\Tr^2[\ETP_s]}{n}+\ric\left(\frac{\d}{\d s}\geo_s,\frac{\d}{\d s}\geo_s\right)\le 0.
\end{equation}
Again, as we saw in  Section \ref{subsec:matrix_disp_cvx_intro}, the entropy $H(\mu_t)$ of a geodesic $(\mu_t)_{t\in [0,1]}$ in $(\mathcal{P}_2(\man), \wasser)$ can be expressed in terms of $\Tr[\ETP_s]$, so \eqref{eq:U_ode_inq_tr} provides the connection between a lower bound on the Ricci curvature and the convexity of   $t\mapsto H(\mu_t)$. The presence of the $\frac{1}{n}$ factor in \eqref{eq:U_ode_inq_tr} is what leads to \emph{dimensional} functional inequalities. In contrast, we will work with \eqref{eq:Us_jacobi}, thus avoiding the dimensional price of Jensen's inequality, which will lead to \emph{intrinsic dimensional} functional inequalities.

\subsubsection{The tensorial Bochner formula}
\label{subsub:Bochner}
The last tool we need from Riemannian geometry is a tensorial version of the classical Bochner formula. The latter formula connects between the derivatives of a function and the Ricci tensor, and is crucial in building the connection between geometry and optimal transport. The \emph{Bochner formula} states that, for any smooth function $\psi:\man\to \R$, we have, for all $x\in \man$,
\begin{equation}
\label{eq:Bochner}
- \frac{1}{2} \Delta | \grad \psi |^{2} + \met(\nabla\psi, \grad \Delta\psi) + \Tr[(\grad^{2} \psi)^{2}] +  \ric(\grad \psi,\grad \psi) = 0,
\end{equation}
where all the terms in \eqref{eq:Bochner} are evaluated at $x$. Since we will be working on the level of the Riemann curvature tensor (sectional curvature), we will use a  tensorial form of \eqref{eq:Bochner}.

\begin{prop}
\label{prop:tensor_bochner}
For any smooth function $\psi:\man\to \R$ we have, for all $x\in \man$,
\begin{equation}
\label{eq:Bochner_tensor}
- \frac{1}{2} \grad^{2} | \grad \psi |^{2} + \grad_{\grad \psi} \grad^{2} \psi + (\grad^{2} \psi)^{2} +  \riem (\grad \psi, \cdot) \grad \psi = 0,
\end{equation}
where all the terms in \eqref{eq:Bochner_tensor} are evaluated at $x$.
\end{prop}
\begin{proof}
The proof of \eqref{eq:Bochner_tensor} is analogous to the way in which \eqref{eq:Bochner} is derived from \eqref{eq:U_ode_tr} \cite[p. 387]{Villani09}, but where we use \eqref{eq:Us_jacobi} instead of  \eqref{eq:U_ode_tr}.  Fix $x\in \man$ and a smooth vector field $\xi$. Define the geodesic $(\geo_s)$ by $\geo_s:=\exp_x(s\xi(x))$ so that $\frac{\d}{\d s}\geo_s|_{s=0}=\xi(x)$. Then, for $s$ small enough,  we may define a vector field $\xi_s$ along the geodesic as the solution of the equation,
\begin{equation}
\label{eq:geodesic_Euler_app}
\frac{\d}{\d s}\geo_s=\xi_s(\geo_s),\qquad \xi_0=\xi.
\end{equation}
Since  $\frac{\D}{\D s}\frac{\d}{\d s}\geo_s=0$ we can differentiate \eqref{eq:geodesic_Euler_app} in $s$ to get
\begin{equation}
\label{eq:Euler_eq}
\left[\frac{\d}{\d s}\xi_s+\grad_{\xi_s}\xi_s\right](\geo_s)=0.
\end{equation}
Fix a basis $\{e_0^i\}_{i=1}^n$ of $T_x\man$, and define the geodesic perturbations $(\geo_s^i(r))$, for $i=1,\ldots, n$, by $\geo_s^i(r):=\exp_{\exp_x(re^i)}(s\xi(\exp_x(re^i)))$. The associated Jacobi fields $(J^1_s,\ldots, J^n_s)$ satisfy  the initial conditions
\begin{equation}
\label{eq:J0}
J^i_0=e_0^i,\qquad \frac{\D }{\D s}J_s^i\bigg|_{s=0} = (\grad\xi)e_0^i.\qquad 
\end{equation}
Taking the $\frac{\D}{\D r}$ derivative of \eqref{eq:geodesic_Euler_app}, applied to $\geo_s(r)$, we get
\begin{equation}
\label{eq:dsDrgeosr}
\frac{\d}{\d s}\frac{\D}{\D r}\geo_s(r)\bigg |_{r=0}=\grad\xi_s(\geo_s(r))\frac{\D}{\D r}\geo_s(r)\bigg |_{r=0},
\end{equation}
and hence the Jacobi matrix satisfies
\begin{equation}
\label{eq:dsJxi}
\dot{\Jf}_s=\nabla\xi_s(\geo_s)\Jf_s,
\end{equation}
where $\grad \xi_s$ is written in the basis $\{e^i_s\}_{i=1}^n$ of $T_{\geo_s}\man$ obtained by parallel transport of $\{e_0^i\}_{i=1}^n$ along $(\geo_s)$. Using the definition \eqref{eq:Us_def} it follows that 
\begin{equation}
\label{eq:U_Euler}
\ETP_s=\grad \xi_s(\geo_s).
\end{equation}
In particular, differentiating \eqref{eq:U_Euler} in $s$, and using \eqref{eq:geodesic_Euler_app} and \eqref{eq:Euler_eq},  yields
\begin{equation}
\label{eq:partialU_Euler}
\dot{\ETP}_s=[-\grad(\grad_{\xi_s}\xi_s)+\grad(\grad\xi_s)\xi_s](\geo_s).
\end{equation}
Combining \eqref{eq:U_Euler} and \eqref{eq:partialU_Euler}, we see that \eqref{eq:Us_jacobi} reads
\begin{equation}
\label{eq:U_ode_euler} 
\left[-\grad(\grad_{\xi_s}\xi_s)+\grad_{\xi_s}\grad\xi_s+(\grad \xi_s)^2+\met\left(\riem\left(\frac{\d}{\d s}\geo_s, \cdot\right)\frac{\d}{\d s}\geo_s, \cdot\right)\right](\geo_s)=0.
\end{equation}
In particular, at $s=0$, \eqref{eq:U_ode_euler}  reads
\begin{equation}
\label{eq:U0_ode_euler} 
[-\grad(\grad_{\xi}\xi)+\grad_{\xi}\grad\xi+(\grad \xi)^2+\met(\riem(\xi, \cdot)\xi, \cdot)](x)=0.
\end{equation}
Taking $\xi=\nabla\psi$ in \eqref{eq:U0_ode_euler}, and using $-\grad(\grad_{\grad\psi}\grad\psi)=- \frac{1}{2} \grad^{2} | \grad \psi |^{2}$, gives \eqref{eq:Bochner_tensor}. 
\end{proof}
 In the classical setting, the Bochner formula leads to the definition of the Bakry--\'Emery operator,
\begin{equation}
\label{eq:gamma_2_def}
\Gamma_{2}(\psi) :=\frac{1}{2} \Delta | \grad \psi |^{2} - \Tr[\grad_{\grad \psi} \grad^{2} \psi]\overset{\eqref{eq:Bochner}}{=}\Tr[(\grad^2\psi)^2]+\ric(\grad\psi,\grad\psi),
\end{equation}
which plays a crucial role in the relations between geometry, optimal transport, and functional inequalities. Analogously, we define the \emph{tensorial Bakry--\'Emery operator},
\begin{equation}
\label{eq:tilde_gamma_2_def}
\tilde{\Gamma}_{2}(\psi) :=\frac{1}{2} \grad^{2} | \grad \psi |^{2} - \grad_{\grad \psi} \grad^{2} \psi\overset{\eqref{eq:Bochner_tensor}}{=} (\grad^{2} \psi)^{2} +  \riem (\grad \psi, \cdot) \grad \psi,
\end{equation}
which will be used in our characterization of nonnegative sectional curvature. 

\subsection{Preliminaries of optimal transport}
\label{subsec:prelim_ot}
In this section we review the necessary preliminaries of optimal transport on Riemannian manifolds. In Section \ref{subsubsec:OT_map} we recall the classical Brenier--McCann theorem on the existence and properties of optimal transport maps on Riemannian manifolds. Using this result we describe in Section \ref{subsub:geodesics} the Wasserstein geodesics, and define our entropy tensor  as well as showing its connection to the entropy functional. In Section \ref{subsubsec:mat_disp_cvx}  we define the notion of matrix displacement convexity and establish some of its properties.
\subsubsection{The optimal transport map}
\label{subsubsec:OT_map}
One of the fundamental results in the theory of optimal transport is the characterization by Brenier, which was extended by McCann to the Riemannian setting, of the deterministic optimal coupling which attains the infimum in the definition of the Wasserstein distance \eqref{eq: w2def_intro}. To describe McCann's result, let $(\man,\met)$ be a smooth  complete Riemannian manifold without boundary, and say that a function $\psi: \man \to \R \cup \{ - \infty \}$ is \emph{$\frac{d^{2}}{2}$-concave} if it is not identically $- \infty$, and there exists a function $\varphi: \man \to \R \cup \{ \pm \infty \}$ such that, for every $x \in \man$,
    \begin{equation}
        \label{eq:c_trasnform_def}
        \psi (x) =  \inf_{y \in \man} \left\{\frac{1}{2} d^{2}(x,y)-\varphi (y)  \right\}.
    \end{equation}
A consequence of the Brenier--McCann theorem  \cite[Theorem 3.2]{Cordero-ErasquinMcCannSchmuckenschlager01}, \cite[Theorems 8 and 9]{McCann01} is that given $\mu_0,\mu_1\in \PR$, there exists a $\frac{d^{2}}{2}$-concave function $-\ot:\man\to  \R \cup \{ + \infty \}$ such that the map $\Ot:\man\to\man$ given by
\begin{equation}
\label{eq:B_Mc_transport}
\Ot(x):=\exp_{x}(\nabla\ot(x)),\qquad\textnormal{for almost-all}\quad x\in\man,
\end{equation}
pushes forward $\mu_0$ to $\mu_1$, and is optimal in the sense that
\begin{equation}
\label{eq:Waserr_Monge}
\wasser^2(\mu_0,\mu_1)=\int_{\man}d^2(x,\Ot(x))\d\mu_0(x). 
\end{equation}
The derivative of $\ot$ exists almost-everywhere in $\man$  \cite[Lemma 3.3]{Cordero-ErasquinMcCannSchmuckenschlager01}, and $\ot$ has a Hessian $\hess\ot$ almost everywhere \cite[Theorem 4.2]{Cordero-ErasquinMcCannSchmuckenschlager01}, which means that, for almost-all $x\in\man$, there exists a symmetric matrix $\hess_x\ot$ on $T_x\man$ such that, for all $u\in T_x\man$ and $\delta>0$,
\begin{equation}
\ot(\exp_x(\delta u))=\ot(x)+\delta\,\met_x(\grad\ot(x),u)+\frac{\delta^2}{2}\,\met_x(u,\hess_x\ot\, u)+o(\delta^2). 
\end{equation}
It follows that, almost-everywhere, $\Ot(x)\notin \cut(x)$, where the cut locus $\cut(x)\subseteq\man$ is the set of points  in $\man$ which cannot be connected to $x$ with a minimizing geodesic \cite[Theorem 4.2]{Cordero-ErasquinMcCannSchmuckenschlager01}.

\subsubsection{Geodesics, Jacobi fields, and entropy tensors}
\label{subsub:geodesics}
The description \eqref{eq:B_Mc_transport} of the optimal transport map provides a description of the geodesics $(\mu_t)_{t\in [0,1]}$ in $(\PR,\wasser)$. Specifically, defining the maps $\Ot_t:\man\to\man$ by
\begin{equation}
\label{eq:B_Mc_transport_t}
\Ot_t(x):=\exp_{x}(t\nabla\ot(x)),\qquad\textnormal{for almost-all}\quad x\in\man,\qquad t\in [0,1],
\end{equation}
we have the relation \cite[Corollary 5.2]{Cordero-ErasquinMcCannSchmuckenschlager01}
\begin{equation}
\label{eq:geo_B_Mc_transport}
\mu_t=(\Ot_t)_{\sharp}\mu_0,\qquad t\in [0,1],
\end{equation}
where the fact that $\mu_t\in \PR$ for all $t\in [0,1]$ follows from \cite[Proposition 5.4]{Cordero-ErasquinMcCannSchmuckenschlager01}. Using Jacobi fields we may relate the densities of $\mu_0$ and $\mu_t$ via the following change of variables formula   \cite[Proposition 2.1]{MR2295207}: Given fixed $t\in [0,1]$  there exists a set $K_t\subseteq \man$ of full $\mu_0$ -measures such that, for each $x\in K_t$, it holds that
\begin{equation}
\label{eq:cov_formula}
\frac{\d \mu_0}{d\vol}(x)=\frac{\d \mu_t}{d\vol}\left(\Ot_t(x)\right)\det (\Jf_t(x)),
\end{equation}
where, given an orthonormal basis $\{e_0^i(x)\}_{i=1}^n$ of $T_x\man$, with $e^1_0(x)\propto \frac{\d}{\d s}\Ot_s |_{s=0}(x)$, $(\Jf_s(x))_{s\in [0,t]}$ is the unique Jacobi field matrix satisfying, for all $s\in [0,t]$,
\begin{equation}
\label{eq:jacobi_mat_ot}
\ddot{\Jf}_s(x)+\Rs_s(x)\Jf_s(x)=0,\qquad \qquad\Jf_0(x)=\Id_{\dd},\qquad \dot{\Jf}_0(x)=\hess_x\ot,
\end{equation}
where $\hess_x\ot$ is represented in the basis $\{e_0^i(x)\}_{i=1}^n$, and where, as in \eqref{eqRs_def}, $\Rs_s(x)$ is an $n\times n$ matrix defined by
\begin{equation}
\label{eqRs_ot_def}
\Rs_s^{ij}(x):=\met_{\Ot_s(x)}\left(\riem_{\Ot_s(x)}\left(\frac{\d}{\d s}\Ot_s(x),e^i_s(x)\right)\frac{\d}{\d s}\Ot_s(x),e^j_s(x)\right),
\end{equation}
with $\{e^i_s(x)\}_{i=1}^n$ being the orthonormal basis of of $T_{\Ot_s(x)}\man$ obtained by parallel transport of $\{e_0^i(x)\}_{i=1}^n$ along the geodesic $s\mapsto \Ot_s(x)$.  By uniqueness we may talk about \emph{the} Jacobi field matrix satisfying \eqref{eq:jacobi_mat_ot}  for all $s\in [0,1]$. Since,  for all $s\in [0,1]$, $\Jf_s(x)$ is invertible \cite[Proposition 2.1]{MR2295207}, we can define, for almost-all $x\in \man$, the symmetric \cite[p. 368]{Villani09} matrix
\begin{equation}
\label{eq:U_def}
\ETP_s(x):=\dot{\Jf}_s(x)\Jf_s^{-1}(x)\qquad\textnormal{for all}\qquad s\in [0,1].
\end{equation}
Arguing as in Section \ref{subsubsec:Jacobi}, we get the equation,  for all $s\in [0,1]$ and almost-all $x\in \man$,
\begin{equation}
\label{eq:U_ODE_ot}
\dot{\ETP}_s(x)+\ETP_s^2(x)+\Rs_s(x)=0,\qquad\qquad \ETP_0(x)=\hess_x\ot.
\end{equation}
Since $\ETP_s(x)$ is only defined for  almost-all $x\in \man$, let us denote, for any $\mu_0,\mu_1\in \PR$, 
\begin{equation}
\label{eq:domain_def}
\dom(\mu_0,\mu_1):=\{x\in \man: \hess_x\ot \textnormal{ exists where $\mu_1=[\exp_{x}(\nabla\ot(x))]_{\sharp}\mu_0$}\}.
\end{equation}
\begin{definition}
\label{def:entropy_tensor}
Given  $\mu_0,\mu_1\in \PR$ we define the \emph{entropy tensor}, at $t\in [0,1]$ and $x\in\dom(\mu_0,\mu_1)$, by setting
\begin{equation}
\label{eq:pointwise_ent_prod_tens_lagrange}
\ET_t^{\mu_0\to\mu_1}(x):=-\int_0^t \ETP_s (x)\d s. 
\end{equation}
\end{definition}

\begin{rem}
\label{rem:ent_tens_ind_meas}
Strictly speaking, the entropy tensor defined in \eqref{eq:pointwise_ent_prod_tens_lagrange} depends only on the vector field $\grad\ot$. However, we always use the entropy tensor with a given source measure $\mu_0$, which in turn determines $\mu_1$ via  the optimal transport map $\exp(\grad\ot)$. Hence, we include $\mu_0$ and $\mu_1$ in the  notation $\ET_t^{\mu_0\to\mu_1}(x)$.
\end{rem}

\begin{rem} \label{rem: tensorsandbasis}
While the definition of the entropy tensor in Definition \ref{def:entropy_tensor} relied on a choice of an orthonormal basis $\{e^i(x)\}_{i=2}^n$ for $x\in \dom(\mu_0,\mu_1)$, our results hold for any such choice. In principle, using the identification of $\Jf_s(x)$ with  $\d \Ot_s(x)$, which is defined in some weak sense \cite[Page 622]{MR2295207}, one could formulate a frame-independent definition of the entropy tensor. However, for the most part it is more convenient to work with frames. 
\end{rem}
The next result expresses the entropy functional along the geodesic $(\mu_t)_{t\in [0,1]}$ in terms of the entropy tensors.
\begin{lem}
\label{lem:tr_ent_mat_is_ent}
Fix $\mu_0,\mu_1\in \PR$ and let $(\mu_t)_{t\in [0,1]}$ be the geodesic in  $\PR$ between $\mu_0$ and $\mu_1$. Then, 
\begin{equation}
\label{eq:tr_ent_mat_is_ent}
H(\mu_t)=H(\mu_0)+\int_{\man}\Tr[\ET_t^{\mu_0\to\mu_1}(x)]\d \mu_0(x).
\end{equation}
\end{lem}
\begin{proof}
Taking the logarithm on both sides of \eqref{eq:cov_formula}, and then integrating against $\mu_0$ over $K_t$, gives
\begin{equation}
\label{eq:cov_entropy_log}
H(\mu_0)=\int_{\man}\log\left(\frac{\d \mu_t}{d\vol}\left(\Ot_t(x)\right)\right)\d\mu_0(x)+\int_{\man}\log \det (\Jf_t(x))\d\mu_0(x),
\end{equation}
where we used that $\mu_0(K_t)=1$ to change the integration domain from $K_t$ to $\man$. By \eqref{eq:geo_B_Mc_transport} we have 
\begin{equation}
\int_{\man}\log\left(\frac{\d \mu_t}{d\vol}\left(\Ot_t(x)\right)\right)\d\mu_0(x)=H(\mu_t),
\end{equation}
so it remains to show that 
\begin{equation}
\label{eq:use_jacobi_formula}
\int_{\man}\log \det (\Jf_t(x))\d\mu_0(x)=-\int_{\man}\Tr[\ET_t^{\mu_0\to\mu_1}(x)]\d \mu_0(x).
\end{equation}
By the Jacobi formula,
\begin{equation}
\label{eq:jacobi_formula}
\begin{split}
\frac{\d}{\d t}\log \det (\Jf_t(x))&=\Tr\left[\dot{\Jf}_t(x)\Jf_t(x)^{-1}\right]\overset{\eqref{eq:U_def}}{=}\Tr[\ETP_t(x)]\overset{\eqref{eq:pointwise_ent_prod_tens_lagrange} }{=}-\frac{\d}{\d t}\Tr[\ET_t^{\mu_0\to\mu_1}(x)],
\end{split}
\end{equation}
so since $\log \det (\Jf_0(x))=0=\Tr[\ET_0^{\mu_0\to\mu_1}(x)]$, \eqref{eq:jacobi_formula} implies that
\begin{equation}
\label{eq:log_Jtr_H}
\log \det (\Jf_t(x))=-\Tr[\ET_t^{\mu_0\to\mu_1}(x)],
\end{equation}
which yields \eqref{eq:use_jacobi_formula}.
\end{proof}

\subsubsection{Matrix displacement convexity}
\label{subsubsec:mat_disp_cvx}

With the definition \eqref{eq:pointwise_ent_prod_tens_lagrange} in hand we define the collection of entropy tensors,
\begin{equation}
\label{eq:H_gen_def}
\ET:=\{(\ET_t^{\mu_0\to\mu_1}(x))_{t\in [0,1],x\in \dom(\mu_0,\mu_1)}\}_{\mu_0,\mu_1\in \PR}.
\end{equation}
The next definition mirrors the classical definition of displacement convexity for the entropy functional. 

\begin{definition}
\label{def:matrix_displacement_cvx}
The entropy tensor $\ET$  is \emph{matrix displacement convex} if for every $\mu_0,\mu_1\in \PR$, $t\in [0,1]$, and $x\in\dom(\mu_0,\mu_1)$, it holds
\begin{equation}
\label{eq:matrix_disp_cvx_def__2nd_der}
\ddot{\ET}_{t}^{\mu_0\to\mu_1}(x) \succeq 0,
\end{equation}
where $\succeq$ is the Loewner order.
\end{definition}
Definition \ref{def:matrix_displacement_cvx} was formulated to mirror the classical definition of displacement convexity for the entropy functional, which can be obtained by taking the trace in \eqref{eq:matrix_disp_cvx_def__2nd_der} and using  Lemma \ref{lem:tr_ent_mat_is_ent}. However, as we saw in Theorem \ref{thm: mainsecequiv_intro}, we will in fact show that when $\man$ has nonnegative sectional curvature we have the inequality stronger than \eqref{eq:matrix_disp_cvx_def__2nd_der}, 
\begin{equation}
\label{eq:matrix_disp_cvx_def__2nd_der_strong}
\ddot{\ET}_{t}^{\mu_0\to\mu_1}(x) \succeq [\dot{\ET}_{t}^{\mu_0\to\mu_1}(x)]^2.
\end{equation}
The next result derives a convexity result for matrices which satisfy inequalities such as \eqref{eq:matrix_disp_cvx_def__2nd_der_strong}. 

\begin{lem}
\label{lem:matrix_dis_equiv}
Fix $T>0$.  Let $\{ \Matrix_t\}_{t\in [0,T]}$ be a family of symmetric linear operators on an inner-product space $(W, \langle\cdot,\cdot\rangle)$ such that $[0,T] \ni t\mapsto\langle w, \Matrix_tw\rangle$ is twice-differentiable for all $w\in W$, and
\begin{equation}
\label{eq:dttA}
\ddot{ \Matrix}_t\succeq [\dot{\Matrix}_t]^2\qquad\textnormal{for all}  \qquad t\in [0,T].
\end{equation}
Then, for every unit vector $w\in W$,
\begin{equation}
\label{eq:e-A}
[0,T]\ni t \mapsto e^{-\langle w, \Matrix_tw\rangle}\textnormal{ is concave.}
\end{equation}
\end{lem}
\begin{proof}
We have
\begin{align*}
\frac{\d}{\d t}e^{-\langle w, \Matrix_tw\rangle}=-e^{-\langle w, \Matrix_tw\rangle}\langle w,\dot{\Matrix}_tw\rangle,
\end{align*}
and 
\begin{align*}
\frac{\d^2}{\d t^2}e^{-\langle w, \Matrix_tw\rangle}=e^{-\langle w, \Matrix_tw\rangle}[\langle w,\dot{ \Matrix}_tw\rangle^2-\langle w,\ddot{\Matrix}_tw\rangle].
\end{align*}
Thus, the concavity of $t \mapsto e^{-\langle w, \Matrix_tw\rangle}$ is equivalent to 
\begin{equation}
\label{eq:2nd_der_conc}
\langle w,\ddot{\Matrix}_tw\rangle\ge \langle w,\dot{\Matrix}_tw\rangle ^2\quad\textnormal{for all}\quad w\in W.
\end{equation}
Jensen's inequality gives $\langle w,[\dot{\Matrix}_t]^2w\rangle\ge \langle w,\dot{\Matrix}_tw\rangle^2$ for every unit vector $w$, which completes the proof. 
\end{proof}

\begin{rem}
The trace analogue of Lemma \ref{lem:matrix_dis_equiv} states that if a twice-differentiable function $A:[0,T]\to \R$ satisfies, for some $c>0$,
\begin{equation}
\label{eq:trace_dtt_A}
\ddot{A}_t\ge c \dot{A}_t^2,
\end{equation}
then 
\begin{equation}
\label{eq:trace_exp_concave}
[0,T]\ni t \mapsto e^{-cA_t}\textnormal{ is concave.}
\end{equation}
In fact, \eqref{eq:trace_exp_concave} is \emph{equivalent} to \eqref{eq:trace_dtt_A}. This is not the case in the tensor case because of the last Jensen inequality in the proof of Lemma \ref{lem:matrix_dis_equiv}. 
\end{rem}

\begin{rem}
\label{rem:matrix_disp_cvx_euclid_def}
The notion of matrix displacement convexity  presented in the flat Euclidean setting \cite{Shenfeld24} was defined by requiring the concavity of 
\begin{equation}
\label{matrix_disp_cvx_def_eulicd}
t\mapsto e^{-\met_x\left(w,\int_{\R^n}\ET_t^{\mu_0\to\mu_1}(x)\d \mu_0(x)w\right)}.
\end{equation}
The condition \eqref{matrix_disp_cvx_def_eulicd} implies the integral analogue of \eqref{eq:matrix_disp_cvx_def__2nd_der},
\begin{equation}
\label{matrix_disp_cvx_def_eulicd_cor}
\frac{\d^2}{\d t^2}\int_{\R^n}\ET_t^{\mu_0\to\mu_1}(x)\d \mu_0(x) \succeq  0.
\end{equation} 
Note that the convexity notion of  \eqref{matrix_disp_cvx_def_eulicd} is  weaker than requiring the concavity of 
\begin{equation}
\label{matrix_disp_cvx_def_eulicd_strong}
t\mapsto e^{-\met_x(w,\ET_t^{\mu_0\to\mu_1}(x)w)}\qquad \textnormal{for all} \quad x,
\end{equation}
and analogously, \eqref{matrix_disp_cvx_def_eulicd_cor} is weaker than requiring \eqref{eq:matrix_disp_cvx_def__2nd_der},
\begin{equation}
\label{matrix_disp_cvx_def_eulicd_srong_cor}
\frac{\d^2}{\d t^2}\ET_t^{\mu_0\to\mu_1}(x) \succeq  0 \qquad \textnormal{for all} \quad x. 
\end{equation}
\end{rem}

\subsection{Main result} 
\label{subsec:main_result}
We are now ready to prove our first main result.
\begin{thm}
\label{thm:main}
    Let $(\man,\met)$ be a smooth complete $n$-dimensional Riemannian manifold without boundary. The following are equivalent. 
    \begin{enumerate}[(i)]
        \item \label{enum: nonnegativesec} The sectional curvature of $\man$ is nonnegative. \\
             \item \label{enum: tensorbochnerlin} For every smooth compactly supported function $\psi:\man\to \R$,
\begin{equation*}
\label{eq:tilde_gamma_2}
\tilde{\Gamma}_{2}(\psi) (x)\succeq 0\qquad \textnormal{for all}\qquad x\in \man.
\end{equation*}
\item \label{enum: tensorbochner} For  every smooth compactly supported  function  $\psi:\man\to \R$,
\begin{equation*}
\label{eq:tilde_gamma_2_inq}
\tilde{\Gamma}_{2}(\psi)(x) \succeq (\grad^{2} \psi)^{2}(x) \qquad \textnormal{for all}\qquad x\in \man.  
\end{equation*}

\item \label{enum: cvx1} For every $\mu_0,\mu_1\in \PR$, $t\in [0,1]$, and $x\in\dom(\mu_0,\mu_1)$,
\begin{equation*}
\label{eq:matrix_disp_cvx_def_2nd_der}
\ddot{\ET}_{t}^{\mu_0\to\mu_1}(x) \succeq [\dot{\ET}_{t}^{\mu_0\to\mu_1}(x)]^2.
\end{equation*}

\item \label{enum: cvx2} For every $\mu_0,\mu_1\in \PR$, $t\in [0,1]$, and $x\in\dom(\mu_0,\mu_1)$, and a unit vector $w\in T_x\man$, the map
\begin{equation*}
\label{eq:matrix_disp_cvx_def}
[0,1]\ni t\mapsto e^{-\met_x(w,\ET_t^{\mu_0\to\mu_1}(x)w)}\quad\textnormal{is concave}.
\end{equation*}
\item \label{enum: cvx3} The entropy tensor $\ET$ is matrix displacement convex, i.e., $\ddot{\ET}_{t}^{\mu_0\to\mu_1}(x)\succeq 0$ for all $t\in [0,1]$ and $x\in\dom(\mu_0,\mu_1)$. 
\end{enumerate}
\end{thm}

\begin{proof}
The implications \eqref{enum: nonnegativesec}$ \Rightarrow$\eqref{enum: tensorbochnerlin}  and  \eqref{enum: nonnegativesec}$\Rightarrow$\eqref{enum: tensorbochner} follow from  \eqref{eq:tilde_gamma_2_def}.  We will show that  \eqref{enum: tensorbochnerlin}$\Rightarrow$\eqref{enum: nonnegativesec} which, together with the trivial implication \eqref{enum: tensorbochner}$\Rightarrow$\eqref{enum: tensorbochnerlin}, will establish
\begin{equation}
\label{eq:123}
\eqref{enum: nonnegativesec}\quad\Longleftrightarrow   \quad  \eqref{enum: tensorbochnerlin}  \quad \Longleftrightarrow \quad  \eqref{enum: tensorbochner}.
\end{equation}
To show \eqref{enum: tensorbochnerlin}$\Rightarrow$\eqref{enum: nonnegativesec} we argue by contradiction. Suppose there exists $x\in\man$, and a plane $P\subseteq T_x\man$ spanned by vectors $u,v\in T_x\man$, such that the sectional curvature $\sec_{x}(P)$ is negative. Let us show how to construct a function $\psi$ such that $\grad \psi (x) = u$ and $\grad^{2} \psi (x) =0$. Choose a small closed ball $B_x\subseteq T_x\man$ of the origin in $T_x\man$ on which the exponential map is injective, and define $f:\exp_x(B_x)\to \R$ by setting, for each  $\exp_{x} (w)\in \exp_x(B_x)$,  $f(\exp_{x} (w)) := \met_x( u, w)$. Per \cite[Lemma 2.26]{Lee13} let $\psi:\man\to \R$ be a smooth compactly supported extension of $f$. By \eqref{eq:tilde_gamma_2_def},
    \begin{equation}
    \begin{split}
\met_x (v , \tilde{\Gamma}_{2} (\psi)(x) v)  &=\met_x ( v , (\grad^{2} \psi)^{2}(x) v ) + \met_x ( v, \riem_x (\grad \psi (x), v) \grad \psi(x)) \\
&=  \met_x( v, \riem_x(u,v)u )= \left( | u |^{2}|v|^{2} - \met_x(u , v )^{2} \right) \sec_{x}(P) < 0,
\end{split}
    \end{equation}
    where the last inequality follows from the Cauchy--Schwarz inequality. This concludes  \eqref{enum: tensorbochnerlin}$\Rightarrow$\eqref{enum: nonnegativesec}, and hence \eqref{eq:123} is established. To complete the proof of the theorem we first note that
 \begin{equation}
\label{eq:456}
\eqref{enum: cvx1}\quad\overset{\textnormal{Lemma \ref{lem:matrix_dis_equiv}}}{\Longrightarrow}\quad\eqref{enum: cvx2},\qquad \eqref{enum: cvx1}\quad\overset{\textnormal{trivial}}{\Longrightarrow}\quad\eqref{enum: cvx3},\qquad \textnormal{and}\qquad \eqref{enum: cvx2}\quad\overset{\textnormal{trivial}}{\Longrightarrow}\quad\eqref{enum: cvx3},
 \end{equation}
so it suffices to show
\begin{equation}
\label{eq:14}
\eqref{enum: nonnegativesec}\quad\Longrightarrow \quad \eqref{enum: cvx1}\qquad\textnormal{and}\qquad \eqref{enum: cvx3} \quad  \Longrightarrow \quad \eqref{enum: nonnegativesec}.
\end{equation}
We begin with $\eqref{enum: nonnegativesec}\Rightarrow\eqref{enum: cvx1}$. Fix  $\mu_0,\mu_1\in \PR$  and let $(\ETP_s(x))_{s\in [0,1]}$ be defined by \eqref{eq:U_def}. Fix $x\in \dom(\mu_0,\mu_1)$. Then, for $t\in [0,1]$, 
\begin{align}
\label{eq:14proof}
\ddot{\ET}_{t}^{\mu_0\to\mu_1}(x)&\overset{\eqref{eq:pointwise_ent_prod_tens_lagrange}}{=}-\dot{\ETP_t}(x)\overset{\eqref{eq:U_ODE_ot}}{=}\ETP_t^2(x)+\Rs_t(x)\overset{\eqref{enum: nonnegativesec}}{\succeq}\ETP_t^2(x)\overset{\eqref{eq:pointwise_ent_prod_tens_lagrange} }{=}[\dot{\ET}_{t}^{\mu_0\to\mu_1}(x)]^2,
\end{align}
which is \eqref{enum: cvx1}.

Finally we prove \eqref{enum: cvx3}$\Rightarrow$\eqref{enum: nonnegativesec}. Suppose for the contrary that there exists $x\in\man$, and a plane $P\subseteq T_x\man$  such that the sectional curvature $\sec_{x}(P)$ is negative. Let  $u,v\in T_x\man$ be two orthonormal vectors which span $P$, and choose the function $f$ as was done in the proof of the implication  \eqref{enum: tensorbochnerlin}$\Rightarrow$\eqref{enum: nonnegativesec}. Denote its extension as $\ot$ so that $\grad\ot(x)=u$ and $\grad^2\ot(x)=0$. Further, by \cite[Theorem 13.5]{Villani09}, we may scale $\ot$ so that $-\ot$ is a $\frac{d^{2}}{2}$-concave function -- note that we no longer have $\nabla \theta(x)= u$ but $\nabla \theta(x) = \epsilon u$ for some $\epsilon >0$. Now take any $\mu_0\in \PR$ and let $\mu_t:=[\exp(t\nabla \ot)]_{\sharp}\mu_0$.  Note that $\dom(\mu_0,\mu_1)=\man$ since $\ot$ is smooth. Let $(\ETP_s(x))_{s\in [0,1]}$ be defined by \eqref{eq:U_def}, and note that, again using $-\grad(\grad_{\grad\ot}\grad\ot)=- \frac{1}{2} \grad^{2} | \grad \ot |^{2}$,
\begin{align}
\label{eq:2nd_time_der_entropy_tildegamma_2_proof}
\begin{split}
0&\overset{ \eqref{enum: cvx3}}{ \preceq}\ddot{\ET}_0^{\mu_0\to \mu_1}(x)\overset{\eqref{eq:pointwise_ent_prod_tens_lagrange}}{=}-\dot{\ETP}_0(x)\overset{ \eqref{eq:partialU_Euler}}{=}\left[ \frac{1}{2} \grad^{2} | \grad \ot |^{2}-\grad_{\grad\ot}\grad^2\ot\right](x)\\
&\overset{\eqref{eq:Bochner_tensor}}{=}[(\grad^{2} \ot)^{2} +  \riem (\grad \ot, \cdot) \grad \ot](x)\overset{\grad^2\ot(x)=0}{=}  [\riem (\grad \ot, \cdot) \grad \ot](x)\overset{\grad\ot(x)=\epsilon u}{=}\epsilon^2 \riem_x(u, \cdot) u.
\end{split}
\end{align}
Taking the inner product with $v$ on both sides of \eqref{eq:2nd_time_der_entropy_tildegamma_2_proof} yields the desired contradiction. 
\end{proof}

\begin{rem}
\label{rem:compact_assumption}
In Theorem \ref{thm:main} we stated that nonnegative sectional curvature implies matrix displacement convexity for measures with compact support. Without the compact support assumption, the Brenier--McCann theorem requires the notion of approximate gradients. However, when $\man$ has nonnegative sectional curvature, approximate gradients can be replaced by true gradients and our proof goes through; see \cite[Theorem 10.41]{Villani09}.
\end{rem}

\begin{rem}
For a general family of matrices $(\Matrix_t)$, Lemma \ref{lem:matrix_dis_equiv} showed that we have the following relations between the convexity conditions of Theorem \ref{thm:main}\eqref{enum: cvx1}-\eqref{enum: cvx2}-\eqref{enum: cvx3},
\begin{equation}
\label{eq:implication_cvx}
\ddot{\Matrix}_t\succeq [\dot{\Matrix}_t]^2\quad \Longrightarrow \quad t \mapsto e^{-\met_x(w,\Matrix_tw)}\textnormal{ is concave}\quad \Longrightarrow \quad \ddot{\Matrix}_t\succeq 0.
\end{equation}
The implications in \eqref{eq:implication_cvx} cannot be reversed in general. However, a corollary of Theorem \ref{thm:main} is that for the entropy tensor all three notions of convexity in \eqref{eq:implication_cvx} are equivalent. 
\end{rem}
\begin{rem}
\label{rem:better}
In the Euclidean setting the work \cite{Shenfeld24} showed that 
\begin{equation}
\label{eq:Euclidean_diff_inq}
\frac{\d^2}{\d t^2}\int_{\R^{\dd}}\ET_{t}^{\mu_0\to\mu_1}(x)\d \mu_0(x) \succeq \left[\frac{\d}{\d t}\int_{\R^{\dd}}\ET_{t}^{\mu_0\to\mu_1}(x)\d \mu_0(x)\right]^2,
\end{equation}
which is weaker than the inequality in Theorem \ref{thm:main}\eqref{enum: cvx1} since by Jensen's inequality 
\begin{equation}
\int_{\R^{\dd}}\left[\dot{\ET}_{t}^{\mu_0\to\mu_1}(x)\right]^2\d \mu_0(x) \succeq 
\left[\int_{\R^{\dd}}\dot{\ET}_{t}^{\mu_0\to\mu_1}(x)\d \mu_0(x)\right]^2.
\end{equation}

\end{rem}

\section{Intrinsic dimensional functional inequalities}
\label{sec:proofApplication}

In this section we derive our intrinsic dimensional functional inequalities. Section \ref{subsec:intrinsi_W_contraction} proves Theorem \ref{thm:intrinsic_EVI_intro} and Corollary \ref{cor:W_contract_intro} and Section \ref{subsec:intrinsicHWI} proves Theorem \ref{thm:intrinsic_HWI_intro}. 

\subsection{Intrinsic dimensional Wasserstein contraction inequalities}
\label{subsec:intrinsi_W_contraction}
We begin with Theorem \ref{thm:intrinsic_EVI_intro}. 
\begin{thm}[Intrinsic dimensional evolution variational inequality]\label{thm:int_dim_EVI}
Let $(\man, \met)$ be a smooth compact $n$-dimensional Riemannian manifold without boundary, and let $(\pheat_t)$ be the  heat semigroup on $\man$. Suppose that the sectional curvature of $\man$ is nonnegative. Then, for all $\mu_0,\mu_1\in \PR$ with $H(\mu_1) < \infty$, and for all $T>0$, we have for almost all $\tau\in (0,T)$,
\begin{equation}
\label{eq:int_dim_EVI}
\frac{\d}{\d \tau}\wasser^{2} (\pheat_{\tau}\mu_{0}, \mu_{1}) \leq 2\int_{\man}\left(n - \Tr[e^{-\ET_1^{\pheat_{\tau}\mu_0\to\mu_1}(x)}]\right)\d \pheat_{\tau}\mu_0(x).
\end{equation}
\end{thm}
\begin{proof}
We follow the usual strategy to establish evolution variational inequalities by using the fact that the heat flow is the gradient flow of the entropy functional over the Wasserstein space. 
\begin{lem}
\label{lem:der_W_H}
Let $(\man, \met)$ and $\mu_0,\mu_1$ be as in Theorem \ref{thm:int_dim_EVI}.  Let $(\mu_{s})_{s \in [0,1]}$ be the Wasserstein geodesic between $\mu_0$ and $\mu_1$, and let $(\mu_s(\tau))_{s\in [0,1]}$ be the Wasserstein geodesic between $\pheat_{\tau}\mu_0$ and $\mu_1$.  Then, given $T>0$, we have, for almost all $\tau \in (0,T)$,
\label{cl:der_entropy_zero}
\begin{equation}
\label{eq:der_entropy_zero}
\frac{1}{2}\frac{\d}{\d \tau}\wasser^{2} (\pheat_{\tau}\mu_{0}, \mu_1)\le\int_{\man}\Tr[\dot{\ET}_0^{\mu_0(\tau)\to\mu_1(\tau)}(x)]\d \mu_0(\tau)(x).
\end{equation}
\end{lem}
\begin{proof}
By \cite[Definition 23.7, Corollary 23.23, and Remark 23.24]{Villani09},
\begin{equation}
\label{eq:H_W}
\frac{1}{2}\frac{\d^+}{\d \tau}\wasser^{2} (\pheat_{\tau}\mu_{0}, \mu_1)\le\frac{\d^+}{\d s}H(\mu_s(\tau))\bigg|_{s=0},
\end{equation}
where for a function $\zeta:[0,\infty)\to \R$ we let the super-derivative be
\begin{equation}
\frac{\d^+}{\d t}\zeta(t):=\limsup_{\delta\downarrow 0}\frac{\zeta(t+\delta)-\zeta(t)}{\delta}.
\end{equation}
We will show that the super-derivatives on both sides of \eqref{eq:H_W} can be replaced by regular derivatives. Given $T>0$ it follows from \cite[Theorem 23.9]{Villani09} that the super-derivative on the left-hand side of \eqref{eq:H_W} can be replaced by a derivative for almost all $\tau\in (0,T)$. 
To analyze  the right-hand side of   \eqref{eq:H_W} we first write, using  $\Tr[\ET_0^{\mu_0(\tau)\to\mu_1(\tau)}(x)]=0$,
\begin{align}
\begin{split}
\label{eq:limsup}
&\frac{\d^+}{\d s}H(\mu_s(\tau))\bigg |_{s=0}=\limsup_{\delta\downarrow 0}\frac{H(\mu_{\delta}(\tau))-H(\mu_0(\tau))}{\delta}\\
&\overset{\eqref{eq:tr_ent_mat_is_ent}
}{=}\limsup_{\delta\downarrow 0}\int_{\man}\frac{\Tr[\ET_{\delta}^{\mu_0(\tau)\to\mu_1(\tau)}(x)]-\Tr[\ET_0^{\mu_0(\tau)\to\mu_1(\tau)}(x)]}{\delta}\d \mu_0(\tau)(x).
\end{split}
\end{align}
Next we want pass the $\limsup$ in \eqref{eq:limsup} inside the integral, to which end we will upper bound the integrand. 
By Theorem \ref{thm:main}\eqref{enum: cvx3},  $s\mapsto \Tr[\ET_s^{\mu_0(\tau)\to\mu_1(\tau)}(x)]$ is convex, so 
\begin{equation}
\label{eq:cvx_tr_H}
\Tr[\ET_{\delta}^{\mu_0(\tau)\to\mu_1(\tau)}(x)]\le (1-\delta)\Tr[\ET_0^{\mu_0(\tau)\to\mu_1(\tau)}(x)]+\delta \Tr[\ET_1^{\mu_0(\tau)\to\mu_1(\tau)}(x)]=\delta \Tr[\ET_1^{\mu_0(\tau)\to\mu_1(\tau)}(x)],
\end{equation}
and hence 
\begin{equation}
\label{eq:tr_H_quotient}
\frac{\Tr[\ET_{\delta}^{\mu_0(\tau)\to\mu_1(\tau)}(x)]-\Tr[\ET_0^{\mu_0(\tau)\to\mu_1(\tau)}(x)]}{\delta}=\frac{\Tr[\ET_{\delta}^{\mu_0(\tau)\to\mu_1(\tau)}(x)]}{\delta}\le \Tr[\ET_1^{\mu_0(\tau)\to\mu_1(\tau)}(x)].
\end{equation}
For each $s\in [0,1]$ the functions $x\mapsto \Tr[\ET_s^{\mu_0(\tau)\to\mu_1(\tau)}(x)]$ are $\mu_0(\tau)$-integrable by Lemma \ref{lem:tr_ent_mat_is_ent}, so by reverse Fatou's Lemma, we conclude from \eqref{eq:limsup} that
\begin{equation}
\begin{split}
\label{eq:limsup_reverse_fatou}
\frac{\d^+}{\d s}H(\mu_s(\tau))\bigg |_{s=0}&\le\int_{\man} \limsup_{\delta\downarrow 0}\frac{\Tr[\ET_{\delta}^{\mu_0(\tau)\to\mu_1(\tau)}(x)]-\Tr[\ET_0^{\mu_0(\tau)\to\mu_1(\tau)}(x)]}{\delta}\d \mu_0(\tau)(x)\\
&=\int_{\man}\Tr[\dot{\ET}_0^{\mu_0(\tau)\to\mu_1(\tau)}(x)]\d \mu_0(\tau)(x),
\end{split}
\end{equation}
where the equality follows from the differentiability of $s\mapsto \Tr[\dot{\ET}_s^{\mu_0(\tau)\to\mu_1(\tau)}(x)]$. The latter differentiability implies that we can replace the $\limsup$  in \eqref{eq:limsup_reverse_fatou} by $\liminf$, so Fatou's Lemma implies
\begin{equation}
\begin{split}
\label{eq:limsup_reverse__direct_fatou}
\frac{\d}{\d s}H(\mu_s(\tau))\bigg |_{s=0}=\int_{\man}\Tr[\dot{\ET}_0^{\mu_0(\tau)\to\mu_1(\tau)}(x)]\d \mu_0(\tau)(x).
\end{split}
\end{equation}
\end{proof}
By Lemma \ref{lem:der_W_H} we now have,  for any measurable choice of orthonormal basis $\{e^i(x)\}_{i=1}^n$ of $T_x\man$, 
\begin{align}
\label{eq:W_der_heat}
\begin{split}
\frac{1}{2}\frac{\d}{\d \tau}\wasser^{2} (\pheat_{\tau}\mu_{0}, \mu_1)&\le \int_{\man}\Tr[\dot{\ET}_0^{\mu_0(\tau)\to\mu_1(\tau)}(x)]\d \mu_0(\tau)(x)\\
&=\sum_{i=1}^n\int_{\man}\met_x\left(\dot{\ET}_0^{\pheat_{\tau}\mu_0\to\mu_1}(x)e^i(x),e^i(x)\right)\d \pheat_{\tau}\mu_0(x).
\end{split}
\end{align}

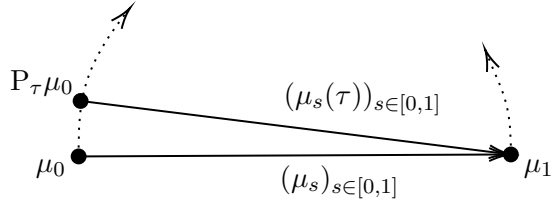
\begin{figure}
\begin{center}
\tikzset{every picture/.style={line width=0.75pt}} 
\begin{tikzpicture}[x=0.75pt,y=0.75pt,yscale=-1,xscale=1]
\draw    (140,177) -- (356,176.01) ;
\draw [shift={(358,176)}, rotate = 179.74] [color={rgb, 255:red, 0; green, 0; blue, 0 }  ][line width=0.75]    (10.93,-3.29) .. controls (6.95,-1.4) and (3.31,-0.3) .. (0,0) .. controls (3.31,0.3) and (6.95,1.4) .. (10.93,3.29)   ;
\draw  [dash pattern={on 0.84pt off 2.51pt}]  (140,177) .. controls (140,149.42) and (140.97,133.48) .. (164.89,103.38) ;
\draw [shift={(166,102)}, rotate = 128.88] [color={rgb, 255:red, 0; green, 0; blue, 0 }  ][line width=0.75]    (10.93,-3.29) .. controls (6.95,-1.4) and (3.31,-0.3) .. (0,0) .. controls (3.31,0.3) and (6.95,1.4) .. (10.93,3.29)   ;

\draw  [dash pattern={on 0.84pt off 2.51pt}]  (358,176) .. controls (358,148.7) and (354.2,146.12) .. (345.67,124.69) ;
\draw [shift={(345,123)}, rotate = 68.63] [color={rgb, 255:red, 0; green, 0; blue, 0 }  ][line width=0.75]    (10.93,-3.29) .. controls (6.95,-1.4) and (3.31,-0.3) .. (0,0) .. controls (3.31,0.3) and (6.95,1.4) .. (10.93,3.29)   ;

\draw  [fill={rgb, 255:red, 0; green, 0; blue, 0 }  ,fill opacity=1 ] (136.5,177) .. controls (136.5,175.07) and (138.07,173.5) .. (140,173.5) .. controls (141.93,173.5) and (143.5,175.07) .. (143.5,177) .. controls (143.5,178.93) and (141.93,180.5) .. (140,180.5) .. controls (138.07,180.5) and (136.5,178.93) .. (136.5,177) -- cycle ;
\draw  [fill={rgb, 255:red, 0; green, 0; blue, 0 }  ,fill opacity=1 ] (354.5,176) .. controls (354.5,174.07) and (356.07,172.5) .. (358,172.5) .. controls (359.93,172.5) and (361.5,174.07) .. (361.5,176) .. controls (361.5,177.93) and (359.93,179.5) .. (358,179.5) .. controls (356.07,179.5) and (354.5,177.93) .. (354.5,176) -- cycle ;
\draw  [fill={rgb, 255:red, 0; green, 0; blue, 0 }  ,fill opacity=1 ] (137.5,149) .. controls (137.5,147.07) and (139.07,145.5) .. (141,145.5) .. controls (142.93,145.5) and (144.5,147.07) .. (144.5,149) .. controls (144.5,150.93) and (142.93,152.5) .. (141,152.5) .. controls (139.07,152.5) and (137.5,150.93) .. (137.5,149) -- cycle ;
\draw    (141,149) -- (356.02,175.75) ;
\draw [shift={(358,176)}, rotate = 187.09] [color={rgb, 255:red, 0; green, 0; blue, 0 }  ][line width=0.75]    (10.93,-3.29) .. controls (6.95,-1.4) and (3.31,-0.3) .. (0,0) .. controls (3.31,0.3) and (6.95,1.4) .. (10.93,3.29)   ;

\draw (117,177) node [anchor=north west][inner sep=0.75pt]    {$\mu _{0}$};
\draw (363,177) node [anchor=north west][inner sep=0.75pt]    {$\mu _{1}$};
\draw (104,132) node [anchor=north west][inner sep=0.75pt]    {$\pheat_{\tau}\mu _{0}$};
\draw (242,138) node [anchor=north west][inner sep=0.75pt]    {$(\mu_s(\tau))_{s\in [0,1]} $};

\draw (240,180) node [anchor=north west][inner sep=0.75pt]    {$(\mu_s)_{s\in [0,1]} $};
\end{tikzpicture}
\end{center}
\caption{Original geodesic $(\mu_s)_{s\in [0,1]}$ between $\mu_0$ and $\mu_1$ and the perturbed geodesic $(\mu_s(\tau))_{s\in [0,1]} $  between $\pheat_{\tau}\mu_0$ and $\mu_1$.}
\end{figure}
By Theorem \ref{thm:main} and Lemma \ref{lem:matrix_dis_equiv}, for any unit vector $w \in T_x\man$, the function 
\begin{equation}
\label{eq:c_def}
c_{w}(s) := e^{- \met_x(w, \ET_s^{\pheat_{\tau}\mu_0\to\mu_1} w)}
\end{equation}
is a concave function of $s$. Thus, 
\begin{equation}
\frac{\d}{\d s}c_{w}(s)\bigg|_{s=0} \geq c_{w}(1) - c_{w}(0),
\end{equation}
which translates to
\begin{equation}
\label{eq:tangent}
- c_{w}(0)  \met_x( w , \dot{\ET}_0^{\pheat_{\tau}\mu_0\to\mu_1} w ) \geq c_{w}(1) - c_{w}(0).
\end{equation}
By the definition \eqref{eq:pointwise_ent_prod_tens_lagrange}, $c_{w}(0) = 1$ so \eqref{eq:tangent} gives 
\begin{equation}
\label{eq:tangen_cort}
\met_x( w , \dot{\ET}_0^{\pheat_{\tau}\mu_0\to\mu_1}\,  w ) \leq 1 - e^{- \met_x(w, \ET_1^{\pheat_{\tau}\mu_0\to\mu_1}  w )}.
\end{equation}
Applying \eqref{eq:tangen_cort} to \eqref{eq:W_der_heat} with $w=e^i(x)$, for an orthonormal basis $\{e^i(x)\}_{i=1}^n$ of $T_x\man$, we get 
\begin{align}
\label{eq:W_der_heat_inq}
\begin{split}
\frac{1}{2}\frac{\d}{\d \tau}\wasser^{2} (\pheat_{\tau}\mu_{0}, \mu_1)\le \int_{\man}\sum_{i=1}^n\left(1 - e^{- \met_x(e^i(x), \ET_1^{\pheat_{\tau}\mu_0\to\mu_1}  e^i(x) )}\right)\d \pheat_{\tau}\mu_0(x).
\end{split}
\end{align}
The next lemma completes the proof of theorem.

\begin{lem}
\label{lem:exp_mat}
    Let $\Matrix$ be a symmetric operator on an $n$-dimensional inner-product space $(W,\langle\cdot,\cdot\rangle)$. Then, 
    \begin{equation}
 \inf \left\{ \sum_{i=1}^n \left(  1- e^{- \langle e^i, \Matrix e^i \rangle} \right):\{e^i\}_{i=1}^n\textnormal{ orthonormal basis of $W$}\right\}  = n - \tr \left( e^{-\Matrix} \right),
    \end{equation}
 where $e^{-\Matrix}$ is the matrix exponential of $\Matrix$.
\end{lem}
\begin{proof}
Given an orthonormal basis $\{e^i\}_{i=1}^n$ we have $e^{- \langle e^i, \Matrix e^i \rangle} =e^{-d_i}$ where $\{d_i\}_{i=1}^n$ are the diagonal entries of $\Matrix$ represented as a matrix with respect to $\{e^i\}_{i=1}^n$. By the Schur--Horn Theorem, the sequence $(d_{1}, \ldots, d_{n})$ is majorized by the sequence $(\lambda_{1}, \ldots , \lambda_{n})$ of eigenvalues of $\Matrix$ (which are independent of the matrix representation of $\Matrix$). The function
    \begin{equation}
(x_{1} , \ldots , x_{n}) \mapsto \sum_{i=1}^n \left(  1 - e^{- x_{i}} \right),
    \end{equation}
    is concave on $\R^n$, and invariant under permutations of the coordinates, so it is Schur-concave. Thus,
    \begin{equation}
\sum_{i=1}^n \left(  1 - e^{- d_{i}} \right)\ge \sum_{i=1}^n \left(  1 - e^{- \lambda_{i}} \right) = n - \tr \left( e^{-\Matrix} \right).
    \end{equation}
Choosing $\{e^i\}_{i=1}^n$ to be an eigenbasis of $\Matrix$ completes the proof.
\end{proof}
\end{proof}
\begin{rem}\label{rem: superderivativeEVI}
In the proof of Lemma \ref{lem:der_W_H}, only the right-hand side of \eqref{eq:H_W} is subjected to manipulation. Thus, \eqref{eq:int_dim_EVI} holds with
$\frac{\d}{\d\tau}$ replaced by $\frac{\d^+}{\d\tau}$, for every
$\tau\in(0,T)$.
\end{rem}
We now derive Corollary \ref{cor:W_contract_intro} below.

\begin{cor}[Intrinsic dimensional Wasserstein contraction along heat flow]
\label{cor:W_contract} 
Let $(\man, \met)$ be a smooth compact $n$-dimensional Riemannian manifold without boundary, and let $(\pheat_t)$ be the  heat semigroup on $\man$. Suppose that the sectional curvature of $\man$ is nonnegative. Then, for all $\mu_0,\mu_1\in \PR$ with $H(\mu_0), H(\mu_1) < \infty$, we have, for any $T>0$, 
\begin{equation}
\label{eq:intrinsic_wass_contract}
 \wasser^{2} (\pheat_T\mu_{0}, \pheat_T\mu_{1})\le  \wasser^{2} (\mu_{0}, \mu_{1})-8\int_0^T\int_{\man} \Tr\left[\sinh^2\left(\frac{\ET_1^{\pheat_{\tau}\mu_0\to\pheat_{\tau}\mu_1}(x)}{2}\right)\right]\d \pheat_{\tau}\mu_0(x)\d \tau.
\end{equation}
\end{cor}
\begin{rem}
In light of the additional integration step mentioned in Remark \ref{rem:matrix_disp_cvx_euclid_def} and  Remark \ref{rem:better}, Theorem \ref{thm:int_dim_EVI} and Corollary \ref{cor:W_contract}  improve (up to the issue of compactness of $\man$, see  Remark \ref{rem:remove_compact}) on \cite[Theorem 5.12 and Theorem 5.13] {Shenfeld24}, respectively, when restricting to optimal transport flows. 
\end{rem}
To prove Corollary \ref{cor:W_contract} we will need the following result on the behavior of  entropy tensors under time-reversal.
\begin{lem}
\label{lem:ent_tens_tim_rever}
Fix $\nu_0,\nu_1\in \PR$ and let $\exp(\grad\ot^{\nu_0\to \nu_1})$ be the optimal transport map from $\nu_0$ to $\nu_1$. Denote
\begin{equation}
\label{eq:Omega}
\tilde{D}(\nu_0,\nu_1):=\{x\in\dom(\nu_0,\nu_1):\exp_x(\grad\ot^{\nu_0\to \nu_1})(x)\in \dom(\nu_1,\nu_0)\}.
\end{equation}
Then, up to a choice of an orthonormal basis, we have
\begin{equation}
\label{eq:ent_tens_tim_rever}
\ET_1^{\nu_1\to\nu_0}(\exp_x(\grad\ot^{\nu_0\to \nu_1}(x)))= -\ET_1^{\nu_0\to\nu_1}(x)\qquad\textnormal{ for almost all } \qquad x\in \tilde{D}(\nu_0,\nu_1).
\end{equation}
\end{lem}
We prove Lemma \ref{lem:ent_tens_tim_rever} below, but first let us show how to complete the proof of Corollary \ref{cor:W_contract}, using the usual strategy (e.g., \cite[Proof of Corollary 13]{MR3952157}) of integrating in time the intrinsic dimensional evolution inequality \eqref{eq:int_dim_EVI}. Let us also note that the set $\tilde{D}(\nu_0,\nu_1)$
has full measure  \cite[Claim 4.4]{Cordero-ErasquinMcCannSchmuckenschlager01}.

\begin{proof}[Proof of Corollary \ref{cor:W_contract}]
By Remark \ref{rem: superderivativeEVI}, the superderivative 
form of \eqref{eq:int_dim_EVI} holds for every positive time. Applying it to the measures $(\mu_1,\mu_0)$ we get 
\begin{align*}
\frac{\d^+}{\d \tau}\wasser^{2} (\pheat_{\tau}\mu_{1}, \mu_{0})& \leq 2\int_{\man}\left(n - \Tr[e^{-\ET_1^{\pheat_{\tau}\mu_1\to\mu_0}(x)}]\right)\d \pheat_{\tau}\mu_1(x)\\
&\overset{\eqref{eq:ent_tens_tim_rever}}{=}2\int_{\man}\left(n - \Tr[e^{\ET_1^{\mu_0\to \pheat_{\tau}\mu_1}(\exp_x(\grad\ot^{\pheat_{\tau}\mu_1\to\mu_0}(x)))}]\right)\d \pheat_{\tau}\mu_1(x)\\
&=2\int_{\man}\left(n - \Tr[e^{\ET_1^{\mu_0\to \pheat_{\tau}\mu_1}(x)}]\right)\d \mu_0(x),
\end{align*}
where $\exp(\grad\ot^{\pheat_{\tau}\mu_1\to\mu_0})$ is the optimal transport map from $\pheat_{\tau}\mu_1$ to $\mu_0$. Hence,
\begin{equation}
\label{eq:int_dim_EVI_opposite}
\frac{\d^+}{\d \tau}\wasser^{2} (\mu_{0}, \pheat_{\tau}\mu_{1}) \leq 2\int_{\man}\left(n - \Tr[e^{\ET_1^{\mu_0\to\pheat_{\tau}\mu_1}(x)}]\right)\d \mu_0(x).
\end{equation}
Fix $s\ge 0$, and apply \eqref{eq:int_dim_EVI_opposite}, with $\pheat_s\mu_{0}$ replacing $\mu_0$, to get
\begin{equation}
\label{eq:s_contract}
\frac{\d^+}{\d \tau} \wasser^{2} (\pheat_s\mu_{0}, \pheat_{\tau}\mu_{1}) \leq 2\int_{\man}\left(n - \Tr[e^{\ET_1^{\pheat_s\mu_0\to\pheat_{\tau}\mu_1}(x)}]\right)\d \pheat_s\mu_0(x).
\end{equation}
Switching between $\mu_0$ and $\mu_1$ in \eqref{eq:s_contract}, and using the symmetry of $\wasser$, we get
\begin{equation}
\label{eq:s_contract_switch}
\frac{\d^+}{\d \tau} \wasser^{2} (\pheat_{\tau}\mu_0, \pheat_s\mu_1) \leq 2\int_{\man}\left(n - \Tr[e^{\ET_1^{\pheat_s\mu_1\to\pheat_{\tau}\mu_0}(x)}]\right)\d \pheat_s\mu_1(x).
\end{equation}
By \cite[Theorem 23.9]{Villani09}, the map $\tau\mapsto
\wasser^2(\pheat_\tau\mu_0,\pheat_\tau\mu_1)$ is locally absolutely continuous and, for almost every $\tau>0$, its
derivative is the sum of the two first variations obtained by evolving
the two arguments separately. By the same pointwise first-variation
estimate used in the proof of Theorem \ref{thm:int_dim_EVI}, these two
terms are bounded by the right-hand sides of
\eqref{eq:s_contract} and \eqref{eq:s_contract_switch}, respectively.
Taking $s=\tau$, we obtain, for almost every $\tau>0$,
\begin{equation}
\label{eq:s_tau}
\begin{split}
&\frac{\d}{\d \tau}\wasser^{2} (\pheat_{\tau}\mu_{0}, \pheat_{\tau}\mu_{1})  \\
&\le 2\int_{\man}\left(n - \Tr[e^{\ET_1^{\pheat_{\tau}\mu_0\to\pheat_{\tau}\mu_1}(x)}]\right)\d \pheat_{\tau}\mu_0(x)+2\int_{\man}\left(n - \Tr[e^{\ET_1^{\pheat_{\tau}\mu_1\to\pheat_{\tau}\mu_0}(x)}]\right)\d \pheat_{\tau}\mu_1(x)\\
&\overset{\eqref{eq:ent_tens_tim_rever}}{=}2\int_{\man}\left(n - \Tr[e^{\ET_1^{\pheat_{\tau}\mu_0\to\pheat_{\tau}\mu_1}(x)}]\right)\d \pheat_{\tau}\mu_0(x)\\
&+2\int_{\man}\left(n - \Tr[e^{-\ET_1^{\pheat_{\tau}\mu_0\to\pheat_{\tau}\mu_1}(\exp_x(\grad\ot^{\pheat_{\tau}\mu_1\to \pheat_{\tau}\mu_0}(x)))}]\right)\d \pheat_{\tau}\mu_1(x)\\
&=2\int_{\man}\left(n - \Tr[e^{\ET_1^{\pheat_{\tau}\mu_0\to\pheat_{\tau}\mu_1}(x)}]\right)\d \pheat_{\tau}\mu_0(x)+2\int_{\man}\left(n - \Tr[e^{-\ET_1^{\pheat_{\tau}\mu_0\to\pheat_{\tau}\mu_1}(x)}]\right)\d \pheat_{\tau}\mu_0(x),
\end{split}
\end{equation}
where $\exp(\grad\ot^{\pheat_{\tau}\mu_1\to \pheat_{\tau}\mu_0})$ is the optimal transport map from $\pheat_{\tau}\mu_1$ to $\pheat_{\tau}\mu_0$. To conclude,
\begin{equation}
\label{eq:s_tau_invert}
\begin{split}
&\frac{\d}{\d \tau} \wasser^{2} (\pheat_{\tau}\mu_{0}, \pheat_{\tau}\mu_{1}) \le 2\int_{\man}\left(2n - \Tr[e^{-\ET_1^{\pheat_{\tau}\mu_0\to\pheat_{\tau}\mu_1}(x)}]-\Tr[e^{\ET_1^{\pheat_{\tau}\mu_0\to\pheat_{\tau}\mu_1}(x)}]\right)\d \pheat_{\tau}\mu_0(x).
\end{split}
\end{equation}
Using the identities $e^x+e^{-x}=2\cosh(x)$ and $2\sinh^2(x)=\cosh(2x)-1$, we have, for any symmetric matrix $A$, 
\begin{equation}
2n-\Tr[e^{-A}]-\Tr[e^{A}]=-4\Tr\left[\sinh^2\left(\frac{A}{2}\right)\right].
\end{equation}
Hence, integrating \eqref{eq:s_tau_invert} from 0 to $T$, we conclude that 
\begin{equation}
\label{eq:s_tau_integrate}
\begin{split}
&\wasser^{2} (\pheat_T\mu_{0}, \pheat_T\mu_{1}) -\wasser^{2} (\mu_{0}, \mu_{1}) \le-8\int_0^T\int_{\man} \Tr\left[\sinh^2\left(\frac{\ET_1^{\pheat_{\tau}\mu_0\to\pheat_{\tau}\mu_1}(x)}{2}\right)\right]\d \pheat_{\tau}\mu_0(x)\d \tau.
\end{split}
\end{equation}
\end{proof}

\begin{proof}[Proof of Lemma \ref{lem:ent_tens_tim_rever}]
In order to understand how entropy tensors behave under time-reversal we need to understand the behavior of Jacobi fields under time-reversal. To this end it will be useful to recall the construction of Jacobi fields from Section \ref{subsub:geodesics}. 

For $x\in \dom(\nu_0,\nu_1)$ define the Jacobi fields $(\Jf_s^{\nu_0\to \nu_1}(x))_{s\in [0,1]}$ associated to $\exp(\grad\ot^{\nu_0\to \nu_1})$ by fixing an orthonormal basis $e_0^{\nu_0\to\nu_1}(x)=\{(e_0^{\nu_0\to\nu_1})^i(x)\}_{i=1}^n$ of $T_x\man$, with $(e^{\nu_0\to\nu_1}_0)^1(x)\propto \frac{\d}{\d s}\Ot_s^{\nu_0\to\nu_1} |_{s=0}(x)$, as the unique solution of
\begin{equation}
\label{eq:jacobi_forward}
\ddot{\Jf}_s^{\nu_0\to \nu_1}(x)+\Rs_s^{\nu_0\to \nu_1}(x)\Jf_s^{\nu_0\to \nu_1}(x)=0,\quad \quad\Jf_0^{\nu_0\to \nu_1}(x)=\Id_{\dd},\quad \dot{\Jf}_0^{\nu_0\to \nu_1}(x)=\hess_x\ot^{\nu_0\to \nu_1},
\end{equation}
where $\hess_x\ot^{\nu_0\to \nu_1}$ is represented in the basis $e_0^{\nu_0\to \nu_1}(x)$, and $(\Rs_s^{\nu_0\to\nu_1}(x))_{s\in [0,1]}$ is defined by
\begin{equation}
\label{eq:Rs_forward}
\begin{split}
&(\Rs_s^{\nu_0\to\nu_1})^{ij}(x):=\\
&\met_{\Ot_s^{\nu_0\to\nu_1}(x)}\left(\riem_{\Ot_s^{\nu_0\to\nu_1}(x)}\left(\frac{\d}{\d s}\Ot_s^{\nu_0\to\nu_1}(x),(e^{\nu_0\to\nu_1}_s)^i(x)\right)\frac{\d}{\d s}\Ot_s^{\nu_0\to\nu_1}(x),(e^{\nu_0\to\nu_1}_s)^j(x)\right), 
\end{split}
\end{equation}
where $e^{\nu_0\to\nu_1}_s(x)$ is the parallel transport of $e^{\nu_0\to\nu_1}_0(x)$ along the geodesic $(\Ot_t^{\nu_0\to \nu_1}(x))_{t\in [0,1]}$,
\begin{equation}
\Ot_t^{\nu_0\to \nu_1}(x):=\exp_{x}(t\nabla\ot^{\nu_0\to \nu_1}(x)),\qquad t\in [0,1].
\end{equation}
We will now define the analogous construction for the optimal transport map $\exp(\grad\ot^{\nu_1\to \nu_0})$ from $\nu_1$ to $\nu_0$. To this end we note that letting, for $y\in \dom(\nu_1,\nu_0)$,
\begin{equation}
\Ot_t^{\nu_1\to \nu_0}(y):=\exp_{y}(t\nabla\ot^{\nu_1\to \nu_0}(y)),\qquad t\in [0,1],
\end{equation}
we have the relation, for almost all\footnote{$\Ot_1^{\nu_1\to \nu_0}$ is the inverse of $\Ot_1^{\nu_0\to \nu_1}$ almost-everywhere  \cite[Page 243]{Cordero-ErasquinMcCannSchmuckenschlager01}.} $x\in \dom(\nu_0,
\nu_1)$, and $t\in [0,1]$,
\begin{equation}
\label{eq:transport_inv_t}
\Ot_t^{\nu_0\to \nu_1}(x) = \Ot_{1-t}^{\nu_1\to \nu_0}(\Ot_1^{\nu_0\to\nu_1}(x)),
\end{equation}
where we used that $[\Ot_1^{\nu_0\to \nu_1}]^{-1}=\Ot_1^{\nu_1\to \nu_0}$  \cite[Page 243]{Cordero-ErasquinMcCannSchmuckenschlager01}. In particular, differentiating \eqref{eq:transport_inv_t} in $t$ gives 
\begin{equation}
\label{eq:transport_inv_t_der}
\frac{\d}{\d t}\Ot_t^{\nu_0\to \nu_1}(x) = -\frac{\d}{\d t}\Ot_{1-t}^{\nu_1\to \nu_0}(\Ot_1^{\nu_0\to\nu_1}(x)).
\end{equation}
To define the Jacobi fields  $(\Jf_r^{\nu_1\to \nu_0}(y))_{r\in [0,1]}$, for  $y=\Ot_1^{\nu_0\to \nu_1}(x)$, we take the orthonormal basis at $T_y\man$ to be $e^{\nu_1\to\nu_0}_0(y):=-e^{\nu_0\to\nu_1}_1(x)$. Note that, as required, 
\begin{equation}
\label{eq:reverse_frame}
(e_0^{\nu_1\to\nu_0})^1(y)=-(e_1^{\nu_0\to\nu_1})^1(x)\propto -\frac{\d}{\d s}\Ot_s^{\nu_0\to\nu_1} \bigg|_{s=1}(x)\overset{\eqref{eq:transport_inv_t_der}}{=}\frac{\d}{\d s}\Ot_s^{\nu_1\to\nu_0} \bigg|_{s=0}(y).
\end{equation}
Further, denoting $e^{\nu_1\to\nu_0}_t(y)$ the parallel transport of $e^{\nu_1\to\nu_0}_0(y)$ along the geodesic $(\Ot_t^{\nu_1\to\nu_0}(y))_{t\in [0,1]}$, we have, by \eqref{eq:transport_inv_t}, 
\begin{equation}
\label{eq:reverse_frame_geo}
e^{\nu_0\to\nu_1}_t(x) =- e^{\nu_1\to\nu_0}_{1-t}(y).
\end{equation}
The Jacobi fields  $(\Jf_r^{\nu_1\to \nu_0}(y))_{r\in [0,1]}$ are the unique solution of
\begin{equation}
\label{eq:jacobi_backward}
\ddot{\Jf}_r^{\nu_1\to \nu_0}(y)+\Rs_r^{\nu_1\to \nu_0}(y)\Jf_r^{\nu_1\to \nu_0}(y)=0,\quad \quad\Jf_0^{\nu_1\to \nu_0}(y)=\Id_{\dd},\quad \dot{\Jf}_0^{\nu_1\to \nu_0}(y)=\hess_y\ot^{\nu_1\to \nu_0},
\end{equation}
where $\hess_y\ot^{\nu_1\to \nu_0}$ is represented in the basis $e_0^{\nu_1\to \nu_0}(y)$, and $(\Rs_r^{\nu_1\to\nu_0}(y))_{r\in [0,1]}$ is given by 
\begin{equation}
\label{eq:Rs_backward}
\begin{split}
&(\Rs_r^{\nu_1\to\nu_0})^{ij}(y):=\\
&\met_{\Ot_r^{\nu_1\to\nu_0}(y)}\left(\riem_{\Ot_r^{\nu_1\to\nu_0}(y)}\left(\frac{\d}{\d r}\Ot_r^{\nu_1\to\nu_0}(y),(e^{\nu_1\to\nu_0}_r)^i(y)\right)\frac{\d}{\d r}\Ot_r^{\nu_1\to\nu_0}(y),(e^{\nu_1\to\nu_0}_r)^j(y)\right)\\
&=\met_{\Ot_{1-r}^{\nu_0\to\nu_1}(x)}\left(\riem_{\Ot_{1-r}^{\nu_0\to\nu_1}(x)}\left(-\frac{\d}{\d r}\Ot_{1-r}^{\nu_0\to\nu_1}(x),-(e^{\nu_0\to\nu_1}_{1-r})^i(x)\right)-\frac{\d}{\d r}\Ot_{1-r}^{\nu_0\to\nu_1}(x),-(e^{\nu_0\to\nu_1}_{1-r})^j(x)\right)\\
&=(\Rs_{1-r}^{\nu_0\to\nu_1})^{ij}(x),
\end{split}
\end{equation}
where we used \eqref{eq:transport_inv_t}, \eqref{eq:transport_inv_t_der}, \eqref{eq:reverse_frame_geo} with $t=1-r$. In order to relate the two Jacobi fields we use the fact \cite[Page 622]{MR2295207} that
\begin{equation}
\label{eq:jacobi_der_transport} 
\begin{split}
&\Jf_s^{\nu_0\to \nu_1}(x)=\d \Ot_s^{\nu_0\to \nu_1}(x),\qquad\textnormal{for all}\qquad x\in \dom(\nu_0,\nu_1)\\
&\Jf_r^{\nu_1\to \nu_0}(y)=\d \Ot_r^{\nu_1\to \nu_0}(y),\qquad\textnormal{for all}\qquad y\in \dom(\nu_1,\nu_0)
\end{split},
\end{equation}
where $\d \Ot_s^{\nu_0\to \nu_1}(x), \d \Ot_r^{\nu_1\to \nu_0}(y)$ are to be understood in a weak sense as in \cite[Page 622]{MR2295207}, and are represented in the respective bases $e_s^{\nu_0\to\nu_1}(x), e_r^{\nu_1\to\nu_0}(y)$. By \cite[Claim 4.3]{Cordero-ErasquinMcCannSchmuckenschlager01}, for any $x\in \dom(\nu_0,\nu_1)$ satisfying $\Ot_1^{\nu_0\to \nu_1}(x)\in \dom(\nu_1,\nu_0)$, we have 
\begin{equation}
\label{eq:inv_func_thm}   
\d \left([\Ot_1^{\nu_0\to \nu_1}]^{-1}\right)(\Ot_1^{\nu_0\to \nu_1}(x)) =[\d\Ot_1^{\nu_0\to \nu_1}(x)]^{-1},
\end{equation}
and hence, by \eqref{eq:jacobi_der_transport},
\begin{equation}
\label{eq:inv_func_thm_jacobi}   
\Jf_1^{\nu_1\to \nu_0}(\Ot_1^{\nu_0\to \nu_1}(x)) =[\Jf_1^{\nu_0\to \nu_1}(x)]^{-1}.
\end{equation}
The following claim uses \eqref{eq:inv_func_thm_jacobi}  to establish the relation between the two Jacobi fields.
\begin{cl}
\label{cl: jacobi_time_reversal}
For almost-all $x\in \dom(\nu_0,\nu_1)$, we have, for all $r\in [0,1]$,
\begin{equation}
\label{eq:jacobi_time_reversal}
\Jf_r^{\nu_1\to \nu_0}(\Ot_1^{\nu_0\to \nu_1}(x))=\Jf_{1-r}^{\nu_0\to \nu_1}(x)\,\Jf_1^{\nu_1\to \nu_0}(\Ot_1^{\nu_0\to \nu_1}(x)).
\end{equation}
\end{cl}
\begin{proof}
Given $x\in \dom(\nu_0,\nu_1)$ let $y:=\Ot^{\nu_0\to \nu_1}_1(x)$, where we note that $x$ and $y$ are not conjugate points, and for $r\in [0,1]$ let
\begin{equation}
\tilde \Jf_r: =\Jf_{1-r}^{\nu_0\to \nu_1}(x)\,\Jf_1^{\nu_1\to \nu_0}(y).
\end{equation}
To show that $\tilde \Jf_r=\Jf_r^{\nu_1\to \nu_0}(y)$ it suffices to show that $(\tilde \Jf_r)_{r\in [0,1]}$ satisfies \eqref{eq:jacobi_backward}, as well as the boundary conditions $\tilde{\Jf}_0=\Jf_0^{\nu_1\to \nu_0}(y)=\Id_n$ and $\tilde{\Jf}_1=\Jf_1^{\nu_1\to \nu_0}(y)$. To verify \eqref{eq:jacobi_backward} we compute
\begin{align*}
\frac{\d^2}{\d r^2}\tilde{\Jf}_r=\ddot{\Jf}_{1-r}^{\nu_0\to \nu_1}(x)\Jf_1^{\nu_1\to \nu_0}(y)\overset{
\eqref{eq:jacobi_forward}}{=}-\Rs_{1-r}^{\nu_0\to \nu_1}(x)\Jf_{1-r}^{\nu_0\to \nu_1}(x)\Jf_1^{\nu_1\to \nu_0}(y)\overset{\eqref{eq:Rs_backward}}{=}-\Rs_{r}^{\nu_1\to \nu_0}(y)\tilde{\Jf}_r(y).
\end{align*}
For the boundary condition at $r=0$ note that 
\begin{equation}
\tilde \Jf_0 =\Jf_{1}^{\nu_0\to \nu_1}(x)\,\Jf_1^{\nu_1\to \nu_0}(y)\overset{\eqref{eq:inv_func_thm_jacobi} }{=}\Id_n,
\end{equation}
while for the boundary condition at $r=1$ we have 
\begin{equation}
\tilde \Jf_1: =\Jf_0^{\nu_0\to \nu_1}(x)\,\Jf_1^{\nu_1\to \nu_0}(y)=\Jf_1^{\nu_1\to \nu_0}(y),
\end{equation}
where we used $\Jf_0^{\nu_0\to \nu_1}(x)=\Id_n$.
\end{proof}
With Claim \ref{cl: jacobi_time_reversal} in hand we can now compute, setting $y:=\Ot_1^{\nu_0\to \nu_1}(x)$ for a fixed $x\in \dom(\nu_0,\nu_1)$,
\begin{equation}
\label{eq:U_reverse}
\begin{split}
&\ETP_r^{\nu_1\to\nu_0}(y)=\dot{\Jf}_r^{\nu_1\to\nu_0}(y)[\Jf_r^{\nu_1\to\nu_0}]^{-1}(y)\\
&=-\dot{\Jf}_{1-r}^{\nu_0\to \nu_1}(x)\,\Jf_1^{\nu_1\to \nu_0}(y)[\Jf_{1-r}^{\nu_0\to \nu_1}(x)\,\Jf_1^{\nu_1\to \nu_0}(y)]^{-1}=-\dot{\Jf}_{1-r}^{\nu_0\to \nu_1}(x)[\Jf_{1-r}^{\nu_0\to \nu_1}(x)]^{-1}\\
&=-\ETP_{1-r}^{\nu_0\to\nu_1}(x).
\end{split}
\end{equation}
Integrating \eqref{eq:U_reverse} in time we get 
\begin{align}
\label{eq:H_time_reverse}
\ET_t^{\nu_1\to\nu_0}(y)&=-\int_0^t \ETP_r^{\nu_1\to\nu_0}(y)\d r=\int_0^t \ETP_{1-r}^{\nu_0\to\nu_1}(x)\d r=\int_{1-t }^1\ETP_r^{\nu_0\to\nu_1}(x)\d r,
\end{align}
and taking $t=1$ in \eqref{eq:H_time_reverse} proves \eqref{eq:ent_tens_tim_rever}.
\end{proof}
\subsection{Intrinsic dimensional HWI inequalities} \label{subsec:intrinsicHWI}
In this section we prove  Theorem \ref{thm:intrinsic_HWI_intro}. 
\begin{thm}[Intrinsic dimensional HWI inequality]
\label{thm:intrinsic_HWI}
Let $(\man,\met)$ be a smooth compact $n$-dimensional Riemannian manifold without boundary, and suppose that the sectional curvature of $\man$ is nonnegative. Assume moreover that $\vol(\man)=1$. Let $\mu\in\PR$ be such that
\begin{equation}
\d\mu=\rho\d\vol,
\end{equation}
where $\rho$ is smooth and strictly positive. Then
\begin{equation}
\label{eq:intrinsicdimHWI}
\int_{\man}
\left(
\Tr\left[e^{-\ET_1^{\mu\to\vol}(x)}\right]-n
\right)
\d\mu(x)
\leq
\wasser(\mu,\vol)\sqrt{I(\mu)}.
\end{equation}
\end{thm}

\begin{proof}
Let $(\mu_t)_{t\in[0,1]}$ be the Wasserstein geodesic from $\mu$ to $\vol$, written as
\begin{equation}
\mu_t=(\Ot_t)_{\#}\mu,
\qquad
\Ot_t(x)=\exp_x(t\grad\ot(x)),
\end{equation}
where $-\ot$ is a $\frac{d^{2}}{2}$-concave function. By Lemma \ref{lem:tr_ent_mat_is_ent}, we have
\begin{equation}
\label{eq:HWI_trace_identity}
H(\mu_t)
=
H(\mu)+
\int_{\man}\Tr\left[\ET_t^{\mu\to\vol}(x)\right]\d\mu(x).
\end{equation}
Therefore, as in Equation \eqref{eq:limsup_reverse__direct_fatou},
\begin{equation}
\label{eq:HWI_entropy_derivative_trace}
\frac{\d}{\d t}\bigg|_{t=0}H(\mu_t)
=
\int_{\man}\Tr\left[\dot{\ET}_0^{\mu\to\vol}(x)\right]\d\mu(x).
\end{equation}

By Theorem \ref{thm: mainsecequiv_intro}, the assumption of nonnegative sectional curvature gives
\begin{equation}
\label{eq:HWI_matrix_diff_ineq}
\ddot{\ET}_t^{\mu\to\vol}(x)
\succeq
\left[\dot{\ET}_t^{\mu\to\vol}(x)\right]^2,
\end{equation}
for all $t\in[0,1]$ and $x\in\dom(\mu,\vol)$. Fix such an $x$, and let $v\in T_x\man$ be a unit vector. By Lemma \ref{lem:matrix_dis_equiv}, the map
\begin{equation}
t\mapsto e^{-\met_x\left(v,\ET_t^{\mu\to\vol}(x)v\right)}
\end{equation}
is concave. Since $\ET_0^{\mu\to\vol}(x)=0$, concavity gives
\begin{equation}
\label{eq:HWI_directional_bound}
e^{-\met_x\left(v,\ET_1^{\mu\to\vol}(x)v\right)}-1
\leq
-\met_x\left(v,\dot{\ET}_0^{\mu\to\vol}(x)v\right).
\end{equation}

Choose an orthonormal basis $e^1,\ldots,e^n$ of $T_x\man$ which diagonalizes
$\ET_1^{\mu\to\vol}(x)$. Applying \eqref{eq:HWI_directional_bound} with
$v=e^i$ and summing over $i=1,\ldots,n$, we obtain
\begin{equation}
\Tr\left[e^{-\ET_1^{\mu\to\vol}(x)}\right]-n
\leq
-\Tr\left[\dot{\ET}_0^{\mu\to\vol}(x)\right].
\end{equation}
Integrating with respect to $\mu$ and using \eqref{eq:HWI_entropy_derivative_trace} gives
\begin{equation}
\label{eq:HWI_before_Cauchy}
\int_{\man}
\left(
\Tr\left[e^{-\ET_1^{\mu\to\vol}(x)}\right]-n
\right)
\d\mu(x)
\leq
-\frac{\d}{\d t}\bigg|_{t=0}H(\mu_t).
\end{equation}

It remains to estimate the initial derivative of the entropy. Applying \cite[Equation (23.28)]{Villani09} with $U(r)=r\log r$
and reference measure $\vol$ in the notation therein (note that we do not have to use approximate gradients due to the compactness assumption on $\man$ (see \cite[Theorem 10.41]{Villani09}, and \cite[Proof of Theorem 23.14: Steps 1-3]{Villani09})), and using that $\rho$ is smooth and strictly positive, gives
\begin{equation}
\frac{\d}{\d t}\bigg|_{t=0}H(\mu_t)
\geq
\int_{\man}
\met_x\left(\grad\ot(x),\grad\log\rho(x)\right)
\d\mu(x),
\end{equation}
and hence
\begin{equation}
-\frac{\d}{\d t}\bigg|_{t=0}H(\mu_t)
\leq
-\int_{\man}
\met_x\left(\grad\ot(x),\grad\log\rho(x)\right)
\d\mu(x).
\end{equation}

Hence, by the Cauchy--Schwarz inequality,
\begin{equation}
\begin{split}
-\frac{\d}{\d t}\bigg|_{t=0}H(\mu_t)
&\leq
\left|
\int_{\man}
\met_x\left(\grad\ot(x),\grad\log\rho(x)\right)
\d\mu(x)
\right|\\
&\leq
\left(\int_{\man}|\grad\log\rho(x)|^2\d\mu(x)\right)^{1/2}
\left(\int_{\man}|\grad\ot(x)|^2\d\mu(x)\right)^{1/2}.
\end{split}
\end{equation}
Now, $\Ot=\exp(\grad\ot)$ is the optimal transport map from $\mu$ to $\vol$. Moreover, almost-everywhere $\Ot(x)\notin \cut(x)$, and consequently, almost-everywhere $|\grad\ot(x)|^2 =  d^{2} ( x, \exp_{x} \grad \theta (x))$. Thus, 
\begin{equation}
\int_{\man}|\grad\ot|^2\d\mu
= \int_{\man}d^{2} ( x, \exp_{x} \grad \theta (x))\d\mu = 
\wasser^2(\mu,\vol).
\end{equation}
Combining this with \eqref{eq:HWI_before_Cauchy} proves \eqref{eq:intrinsicdimHWI}.
\end{proof}

\begin{cor}[Dimensional HWI inequality]
\label{cor:dimensional_HWI}
Under the assumptions of Theorem \ref{thm:intrinsic_HWI}, we have
\begin{equation}
\label{eq:dimensional_HWI}
n\left(e^{\frac{1}{n}H(\mu)}-1\right)
\leq
\wasser(\mu,\vol)\sqrt{I(\mu)}.
\end{equation}
\end{cor}

\begin{proof}
By Jensen's inequality applied to the eigenvalues of $\ET_1^{\mu\to\vol}(x)$,
\begin{equation}
\frac{1}{n}
\Tr\left[e^{-\ET_1^{\mu\to\vol}(x)}\right]
\geq
\exp\left(
-\frac{1}{n}
\Tr\left[\ET_1^{\mu\to\vol}(x)\right]
\right).
\end{equation}
Applying Jensen's inequality once more, now in the variable $x$, gives
\begin{equation}
\frac{1}{n}
\int_{\man}
\Tr\left[e^{-\ET_1^{\mu\to\vol}(x)}\right]\d\mu(x)
\geq
\exp\left(
-\frac{1}{n}
\int_{\man}
\Tr\left[\ET_1^{\mu\to\vol}(x)\right]\d\mu(x)
\right).
\end{equation}
Using the trace identity \eqref{eq:tensor_ent_trace_intro} at $t=1$, and the fact that $H(\vol)=0$, we have
\begin{equation}
\int_{\man}
\Tr\left[\ET_1^{\mu\to\vol}(x)\right]\d\mu(x)
=
H(\vol)-H(\mu)
=
-H(\mu).
\end{equation}
Hence
\begin{equation}
\int_{\man}
\Tr\left[e^{-\ET_1^{\mu\to\vol}(x)}\right]\d\mu(x)
\geq
n e^{\frac{1}{n}H(\mu)}.
\end{equation}
Therefore
\begin{equation}
n\left(e^{\frac{1}{n}H(\mu)}-1\right)
\leq
\int_{\man}
\left(
\Tr\left[e^{-\ET_1^{\mu\to\vol}(x)}\right]-n
\right)
\d\mu(x).
\end{equation}
Combining this with Theorem \ref{thm:intrinsic_HWI} proves \eqref{eq:dimensional_HWI}.
\end{proof}
\section{Intrinsic dimensional contraction characterizes nonnegative curvature} \label{sec:converse}
In this section we prove Theorem \ref{thm:W_contract_char_sec_intro}. The proof of the converse result is by contradiction where we assume that there is a tangent plane at a point $p$ with negative curvature. Broadly speaking, this assumption is then used to contradict linear-algebraic inequalities obtained for $T_{p} \man$  by localizing \eqref{eq:intrinsic_wass_contract_intro}. 

\subsection{Preliminaries}
An elementary but key technical tool in the proof of the converse result is a small-time estimate for Jacobi matrices presented below. 

\begin{lem}
\label{lem:small_speed_entropy_tensor_expansion}
Let $p\in \man$, let $u\in T_p\man$ be a unit vector, and define the linear operator $R_u:T_p\man\to T_p\man$ by
\begin{equation}
R_uv:=\riem_p(u,v)u,\qquad v\in T_p\man.
\end{equation}
Let $e^1,\ldots,e^n$ be an orthonormal basis of $T_p\man$ diagonalizing $R_u$, with $e^1=u$, and let $\{\kappa_i\}_{i=1}^n$ be the associated eigenvalues,
\begin{equation}
R_u e^i=\kappa_i e^i,\qquad i=1,\ldots,n.
\end{equation}
Let $A:T_p\man\to T_p\man$ be a linear operator which is also diagonal in  $e^1,\ldots,e^n$, with the associated eigenvalues $\{a_i\}_{i=1}^n$,
\begin{equation}
Ae^i=a_i e^i, \qquad \textnormal{satisfying}\qquad \textnormal{$a_i>-1$}\quad \textnormal{for $i=1,\ldots, n$.}
\end{equation}
For $\epsilon>0$ let $\Jf_s^\epsilon$ be the Jacobi field matrix along the family of geodesics
\begin{equation}
\geo_s^{\epsilon}:=\exp_p(s\epsilon u),\qquad s\in[0,1],
\end{equation}
written in the parallel frame induced by $e^1,\ldots,e^n$, with
\begin{equation}
\Jf_0^\epsilon:=\Id,\qquad \dot{\Jf}_0^\epsilon:=A.
\end{equation}
Set
\begin{equation}
\ETP_s^\epsilon:=\dot{\Jf}_s^\epsilon(\Jf_s^\epsilon)^{-1},
\qquad
\ET_1^\epsilon:=-\int_0^1\ETP_s^\epsilon\d s.
\end{equation}
Then, as $\epsilon\to0$,
\begin{equation}
\Tr[A-\ETP_1^\epsilon] -4\Tr\left[\sinh^2\left(\frac{\ET_1^\epsilon}{2}\right)\right]=\epsilon^2\sum_{i=1}^n\kappa_i\frac{a_i^2+6a_i+6}{6(1+a_i)}+O(\epsilon^3).
\end{equation}
\end{lem}

\begin{proof}
Let $\Rs_s^\epsilon$ be the curvature matrix along $\geo_s^\epsilon$, written in the parallel frame induced by $e^1,\ldots,e^n$, as in Section \ref{subsubsec:Jacobi}, so that 
\begin{equation}
\ddot{\Jf}_s^\epsilon+\Rs_s^\epsilon\Jf_s^\epsilon=0.
\end{equation}
Throughout, we will identify $R_u$ with its matrix in the basis $e^1,\ldots,e^n$, so that it may be directly compared with the curvature matrix $\Rs_s^\epsilon$. Equivalently, this is the comparison obtained by identifying tangent spaces along $\geo^\epsilon$ using parallel transport. 

Written as matrices in the parallel frame along $\geo^\epsilon$, the expression
$\riem_{\geo_s^\epsilon}(u_s,\cdot)u_s$ is a matrix-valued function depending smoothly on $(\epsilon,s)$ for $\epsilon$ small and $s\in[0,1]$, where $(u_s)$ is the parallel transport of $u$ along $(\gamma_s^{\epsilon})$ (note that all the $\gamma^{\epsilon}$'s are just extensions of $\gamma^{\epsilon'}$ when $\epsilon' < \epsilon$). This is so because the Riemann curvature tensor is smooth, and parallel transport, being a solution to a certain linear ODE, depends smoothly on its initial conditions (see, for example, \cite[Theorem 4.31]{Lee18book}). 
Since $\geo^\epsilon$ has speed $\epsilon$, the segment $\{\geo_s^\epsilon:0\leq s\leq1\}$ has length at most $\epsilon$.
The smoothness therefore gives a uniform Lipschitz bound on this family of matrices. Hence, uniformly in $s\in[0,1]$, when written as matrices,
\begin{equation}
\riem_{\geo_s^\epsilon}(u_s,\cdot)u_s
=
\riem_p(u,\cdot)u+O(\epsilon).
\end{equation}
Keeping in mind $\Rs^{\epsilon}_{s} = \riem_{\geo_s^\epsilon}(\epsilon u_s,\cdot) \epsilon u_s = \epsilon^2 \riem_{\geo_s^\epsilon}(u_s,\cdot)u_s$, and $\riem_p(u,\cdot)u = R_u$, we get 
\begin{equation}
\Rs_s^\epsilon=\epsilon^2R_u+O(\epsilon^3),
\end{equation}
uniformly in $s\in[0,1]$. Thus,
\begin{equation} \label{eq: JacobieqnwithRu}
\ddot{\Jf}_s^\epsilon+\epsilon^2R_u\Jf_s^\epsilon=O(\epsilon^3).
\end{equation}
Plugging in $\epsilon = 0$ in the above equation, we see that $\ddot{\Jf}_s^0=0$, which together with the initial conditions this gives
\begin{equation}
\Jf_s^0=\Id+sA.
\end{equation}
Further, write
\begin{equation}
\Jf_s^\epsilon=\Jf_s^0+\epsilon L_s+O(\epsilon^2)
\end{equation}
as a Taylor expansion in $\epsilon$, for some $L_{s}$. Substituting this into \eqref{eq: JacobieqnwithRu} and comparing the terms of order $\epsilon$ gives $\ddot L_s=0$. Since the initial conditions for $\Jf_s^\epsilon$ are independent of $\epsilon$, we also have $L_0=0$ and $\dot L_0=0$. Hence $L_s=0$ for all $s\in[0,1]$. Thus,
\begin{equation} \label{eq: J=I+A&some}
\Jf_s^\epsilon=\Id+sA+\epsilon^2K_s+O(\epsilon^3),
\end{equation} 
for some $K_s$. Substituting this expansion into the Jacobi equation and comparing the terms of order $\epsilon^2$ gives
\begin{equation}
\ddot K_s=-R_u(\Id+sA),
\qquad
K_0=0,
\qquad
\dot K_0=0.
\end{equation}
Since $A$ and $R_u$ are diagonal in the chosen basis, so is $K_s$. Writing $K_s e^i=k_i(s)e^i$, we get
\begin{equation}
\ddot k_i(s)=-\kappa_i(1+a_i s),
\qquad
k_i(0)=0,
\qquad
\dot k_i(0)=0.
\end{equation}
Therefore
\begin{equation}
k_i(s)
=
-\kappa_i
\left(
\frac{s^2}{2}+\frac{a_i s^3}{6}
\right).
\end{equation}

Set $D_s:=\Id+sA$. From the expansion of $\Jf_s^\epsilon$,
\begin{equation}
(\Jf_s^\epsilon)^{-1}
=
D_s^{-1}
-
\epsilon^2D_s^{-1}K_sD_s^{-1}
+
O(\epsilon^3),
\end{equation}
and consequently,
\begin{equation}
\ETP_s^\epsilon
=
AD_s^{-1}
+
\epsilon^2
\left(
\dot K_sD_s^{-1}
-
AD_s^{-1}K_sD_s^{-1}
\right)
+
O(\epsilon^3).
\end{equation}
Taking the $i$-th diagonal entry gives
\begin{equation}
(\ETP_s^\epsilon)_{ii}
=
\frac{a_i}{1+a_i s}
+
\epsilon^2
\left(
\frac{\dot k_i(s)}{1+a_i s}
-
\frac{a_i k_i(s)}{(1+a_i s)^2}
\right)
+
O(\epsilon^3).
\end{equation}
At $s=1$,
\begin{equation}
k_i(1)
=
-\kappa_i\left(\frac12+\frac{a_i}{6}\right),
\qquad
\dot k_i(1)
=
-\kappa_i\left(1+\frac{a_i}{2}\right),
\end{equation}
and hence
\begin{equation}
\label{eq:U1_exp} 
(\ETP_1^\epsilon)_{ii}
=
\frac{a_i}{1+a_i}
+
\epsilon^2\kappa_i
\left(
-
\frac{1+\frac{a_i}{2}}{1+a_i}
+
\frac{a_i\left(\frac12+\frac{a_i}{6}\right)}{(1+a_i)^2}
\right)
+
O(\epsilon^3).
\end{equation}
Moreover,
\begin{equation}
\frac{\dot k_i(s)}{1+a_i s}
-
\frac{a_i k_i(s)}{(1+a_i s)^2}
=
\frac{\d}{\d s}
\left(
\frac{k_i(s)}{1+a_i s}
\right),
\end{equation}
and therefore
\begin{equation}
\label{eq:trace_expan}
(\ET_1^\epsilon)_{ii}
=
-\log(1+a_i)
+
\epsilon^2\kappa_i
\frac{\frac12+\frac{a_i}{6}}{1+a_i}
+
O(\epsilon^3).
\end{equation}
Next we will use the elementary trace expansion
\begin{equation}
\label{eq:elment_trac_exp}
\Tr[h(B+\epsilon^2 C+O(\epsilon^3))]
=
\Tr[h(B)]
+
\epsilon^2\Tr[h'(B)C]
+
O(\epsilon^3),
\end{equation}
valid for smooth $h$ and symmetric matrices $B,C$. We will apply \eqref{eq:elment_trac_exp} to \eqref{eq:trace_expan} with $h(x):=4\sinh^2(\frac{x}{2})$  and $B,C$ the diagonal matrices whose entries are $\{-\log(1+a_i)\}_{i=1}^n$ and $\left\{\kappa_i
\frac{\frac12+\frac{a_i}{6}}{1+a_i}\right\}_{i=1}^n$, respectively. To this end we note that
\[
h(-\log(1+a_i))=\frac{a_i^2}{1+a_i},\qquad h'(-\log(1+a_i))=\frac{-a_i(a_i+2)}{1+a_i},
\]
so that 
\begin{equation}
4\Tr\left[\sinh^2\left(\frac{\ET_1^\epsilon}{2}\right)\right]=\sum_{i=1}^n\left[\frac{a_i^2}{1+a_i}-\epsilon^2\frac{a_i(a_i+2)}{1+a_i}\kappa_i
\frac{\frac12+\frac{a_i}{6}}{1+a_i}\right]+O(\epsilon^3).
\end{equation}
Up to an $O(\epsilon^3)$ error, the matrices $\ETP_1^\epsilon$ and $\ET_1^\epsilon$ are diagonal in the chosen basis, so using \eqref{eq:U1_exp} we get
\begin{equation}
    \begin{split}
   & \Tr[A-\ETP_1^\epsilon] -4\Tr\left[\sinh^2\left(\frac{\ET_1^\epsilon}{2}\right)\right]=\sum_{i=1}^n\left[a_i-\frac{a_i}{1+a_i} -\frac{a_i^2}{1+a_i}\right]\\
&+\epsilon^2\sum_{i=1}^n\left[-\kappa_i
\left(
-
\frac{1+\frac{a_i}{2}}{1+a_i}
+
\frac{a_i\left(\frac12+\frac{a_i}{6}\right)}{(1+a_i)^2}\right)+\frac{a_i(a_i+2)}{1+a_i}\kappa_i
\frac{\frac12+\frac{a_i}{6}}{1+a_i}\right]+O(\epsilon^3)\\
&=\epsilon^2\sum_{i=1}^n\kappa_i\frac{a_i^2+6a_i+6}{6(1+a_i)}+O(\epsilon^3).
\end{split}
\end{equation}

\end{proof}
Next, we need a lemma which allows the existence of optimal transport maps with desired properties. We state it without a complete proof since its proof is straightforward and similar to the arguments used in the proof of Theorem \ref{thm:main}.
\begin{lem}
\label{lem:local_realization_transport_jets}
Let $p\in \man$, let $v\in T_p\man$, and let $A:T_p\man\to T_p\man$ be symmetric with
\begin{equation}
\Id+A\succ0.
\end{equation}
Then, after possibly replacing $v$ by a sufficiently small multiple of itself, there exist a ball $B_p$ centered at $p$ of radius $r>0$, and a smooth function $\ot$ on $B_p$, such that
\begin{equation}
\grad\ot(p)=v,
\qquad
\hess_p\ot=A,
\end{equation}
and such that
\begin{equation}
\Ot(x):=\exp_x(\grad\ot(x))
\end{equation}
is the optimal transport map from $\mu$ to $\Ot_{\#}\mu$ for every $\mu\in\PR$ supported in $B_p$.
\end{lem}
\begin{proof}[Proof idea]
Define $\theta (\exp_{x} w) = \langle v , w \rangle + \frac{1}{2} A(w,w) + O(\vert w \vert^3)$ in a small ball around $p$. In fact, for the $O(\vert w \vert^3)$ term one could explicitly take $|w|^4$ for small $|w| < 1$. Since $\Id+A\succ0$, after shrinking the ball and taking $|v|$ sufficiently small, $-\ot$ is $\frac{d^2}{2}$-concave. Thus the map $\Ot=\exp(\grad\ot)$ is optimal for every absolutely continuous probability measure supported in that ball.
\end{proof}
\subsection{Main result}

\begin{thm}[Intrinsic dimensional Wasserstein contraction characterizes nonnegative sectional curvature]
\label{thm:converse_intrinsic_W_contract}
Let $(\man,\met)$ be a smooth compact $n$-dimensional Riemannian manifold without boundary. Suppose that, for every $\mu_0,\mu_1\in\PR$ and every $T>0$, we have
\begin{equation}
\label{eq:converse_intrinsic_W_contract}
 \wasser^{2}(\pheat_T\mu_{0},\pheat_T\mu_{1})
\leq
\wasser^{2}(\mu_{0},\mu_{1})
-
8\int_0^T\int_{\man}
\Tr\left[
\sinh^2\left(
\frac{\ET_1^{\pheat_{\tau}\mu_0\to\pheat_{\tau}\mu_1}(x)}{2}
\right)
\right]
\d\pheat_{\tau}\mu_0(x)\d\tau.
\end{equation}
Then, $\man$ has nonnegative sectional curvature. 
\end{thm}

\begin{proof}
Suppose, towards a contradiction, that the sectional curvature is negative at some point $p\in\man$. Then, there exist a unit vector $u\in T_p\man$, and an orthonormal basis $e^1,\ldots,e^n$ of $T_p\man$, with $e^1=u$, such that, for $i=1,\ldots n$,
\begin{equation}
R_u e^i=\kappa_i e^i,
\qquad
\kappa_1=0,
\qquad
\kappa_n<0
\end{equation}
for some scalars $\{\kappa_i\}_{i=1}^n$; here $\kappa_1=0$ because $R_{u}u = 0$ which follows from the symmetry properties of the Riemann tensor. Fix $\lambda>0$, to be chosen later, and define $A:T_p\man\to T_p\man$ by
\begin{equation}
Ae^n=\lambda e^n,
\qquad
Ae^i=0
\quad\textnormal{for }i=1,\ldots,n-1.
\end{equation}
For all sufficiently small $\epsilon>0$, choose a smooth optimal transport potential $\ot$ satisfying 
\begin{equation}
\grad\ot(p)=\epsilon u,
\qquad
\hess_p\ot=A.
\end{equation}
Note that the existence of these potentials $\theta$, which are of course different for different $\epsilon$, is guaranteed by Lemma \ref{lem:local_realization_transport_jets}.
Fix $r>0$ and let $B_p$ be a ball of radius $r$ around $p$. Let $(\mu_0^\delta)_{\delta>0}\subseteq\PR$ be a family of smooth probability measures concentrating at $p$ as $\delta\downarrow0$, with, for all $\delta>0$,
$\supp(\mu_0^\delta)\subset B_p$. For $t\in [0,1]$ let
\begin{equation}
\mu_t^\delta:=(\Ot_t)_{\#}\mu_0^\delta,
\qquad
\Ot_t(x):=\exp_x(t\grad\ot(x)),
\end{equation}
be the Wasserstein geodesic from $\mu_0^\delta$ to $\mu_1^\delta$.

We apply \eqref{eq:converse_intrinsic_W_contract} to the pair $(\mu_0^\delta,\mu_1^\delta)$. Since the two sides of the inequality have the same value at $T=0$, we differentiate the inequality at $T=0$, to obtain
\begin{equation}
\label{eq:_local_sec_Equiv}
\frac{\d}{\d T}\bigg|_{T=0}
\wasser^2(\pheat_T\mu_0^\delta,\pheat_T\mu_1^\delta)
\leq
-8
\int_{\man}
\Tr\left[
\sinh^2\left(
\frac{\ET_1^{\mu_0^\delta\to\mu_1^\delta}(x)}{2}
\right)
\right]
\d\mu_0^\delta(x).
\end{equation}
Let us rewrite the left-hand side of \eqref{eq:_local_sec_Equiv}.
\begin{cl}
\label{cl:_local_sec_Equiv_lhs}
\begin{equation}
\label{eq:_local_sec_Equiv_lhs}
\frac12
\frac{\d}{\d T}\bigg|_{T=0}
\wasser^2(\pheat_T\mu_0^\delta,\pheat_T\mu_1^\delta)= \int_{\man}
\Tr\left[
\ETP_1^{\mu_0^\delta\to\mu_1^\delta}(x)
-
\ETP_0^{\mu_0^\delta\to\mu_1^\delta}(x)
\right]
\d\mu_0^\delta(x).
\end{equation}
\end{cl}
\begin{proof}[Proof of Claim \ref{cl:_local_sec_Equiv_lhs}]
By \cite[Theorem 23.9]{Villani09},
\begin{equation}
\begin{split}
&\frac12
\frac{\d}{\d T}
\wasser^2(\pheat_T\mu_0^\delta,\pheat_T\mu_1^\delta)=\\
&\int_{\man}\met_x\left(\nabla\psi_T^0(x),\nabla\log \pheat_T\mu_0^\delta(x)\right)\d\pheat_T\mu_0^\delta(x)+\int_{\man}\met_x\left(\nabla\psi_T^1(x),\nabla\log \pheat_T\mu_1^\delta(x)\right)\d\pheat_T\mu_1^\delta(x)
\end{split}
\end{equation}
where $\exp(\nabla\psi_T^0)_{\sharp}\pheat_T\mu_0^\delta=\pheat_T\mu_1^\delta$ and $\exp(\nabla\psi_T^1)_{\sharp}\pheat_T\mu_1^\delta=\pheat_T\mu_0^\delta$ are the optimal transport maps (we again use \cite[Theorem 10.41]{Villani09}), so in particular,
\begin{equation}
\begin{split}
\label{eq:wass_grad_ent}
&\frac12
\frac{\d}{\d T}\bigg|_{T=0}
\wasser^2(\pheat_T\mu_0^\delta,\pheat_T\mu_1^\delta)=\\
&\int_{\man}\met_x\left(\nabla\psi_0^0(x),\nabla\log \mu_0^\delta(x)\right)\d\mu_0^\delta(x)+\int_{\man}\met_x\left(\nabla\psi_0^1(x),\nabla\log \mu_1^\delta(x)\right)\d\mu_1^\delta(x),
\end{split}
\end{equation}
By integration by parts, interpreted in the sense of distributions,
\begin{equation}
\begin{split}
\int_{\man}\met_x\left(\nabla\psi_0^0(x),\nabla\log \mu_0^\delta(x)\right)\d\mu_0^\delta(x)&=-\int_{\man}\Delta\psi_0^0(x)\d\mu_0^\delta(x)=\frac{\d}{\d t}\bigg|_{t=0}H(\mu_t^\delta)\\
&=-\int_{\man}\tr\left[\ETP_0^{\mu_0^\delta\to\mu_1^\delta}(x)\right]\d\mu_0^\delta(x),
\end{split}
\end{equation}
where the second identity holds by \cite[Equation 23.29]{Villani09}, and the last identity by Lemma \ref{lem:tr_ent_mat_is_ent}. Similarly, for the second term in \eqref{eq:wass_grad_ent}, we first note that $\mu_1^\delta$ is smooth since $\mu_1^\delta=(\Ot_1)_{\#}\mu_0^\delta$ and $\Ot_1$ can be made a diffeomorphism by taking $\epsilon$ small and shrinking its domain to a smaller ball if necessary. Indeed, $\d\Ot_1(p) = \Jf^{\epsilon}_{1} = \Id + A + O(\epsilon^2)$ by equations \eqref{eq:jacobi_der_transport} and \eqref{eq: J=I+A&some}, and thus $\d\Ot_1(p)\succ 0$ for small $\epsilon$ given we have $I + A \succ 0$, thereby allowing an application of the inverse function theorem.   Hence,
\begin{equation}
\int_{\man}\met_x\left(\nabla\psi_0^1(x),\nabla\log \mu_1^\delta(x)\right)\d\mu_1^\delta(x)=-\int_{\man}\Tr\left[\ETP_0^{\mu_1^\delta\to\mu_0^\delta}(x)\right]\d\mu_1^\delta(x)=\int_{\man}\tr\left[\ETP_1^{\mu_0^\delta\to\mu_1^\delta}(x)\right]\d\mu_0^\delta(x),
\end{equation}
where the last equality follows from \eqref{eq:U_reverse}.
\end{proof}
Since $\ETP_0^{\mu_0^\delta\to\mu_1^\delta}=\hess\ot$ (cf. Equation \eqref{eq:U_ODE_ot}) we obtain by \eqref{eq:_local_sec_Equiv}-\eqref{eq:_local_sec_Equiv_lhs},
\begin{equation}
\int_{\man}
\left(
\Tr\left[
\hess_x\ot-\ETP_1^{\mu_0^\delta\to\mu_1^\delta}(x)
\right]
-
4\Tr\left[
\sinh^2\left(
\frac{\ET_1^{\mu_0^\delta\to\mu_1^\delta}(x)}{2}
\right)
\right]
\right)
\d\mu_0^\delta(x)
\geq0.
\end{equation}
It may be helpful to observe here that varying $\delta$ leaves quantities such as $\ET_1^{\mu_0^\delta\to\mu_1^\delta}(x)$ unchanged in their domain of definition, as they solely depend on the transport maps and not the measures despite our notation; cf. Remark \ref{rem:ent_tens_ind_meas}. Letting $\delta\downarrow0$, and using the smooth dependence of the Jacobi field matrices on the initial point, gives the pointwise inequality at $p$,
\begin{equation} \label{eq:pointwiseineqforWasserContraction}
\Tr[A-\ETP_1^\epsilon]
-
4\Tr\left[
\sinh^2\left(
\frac{\ET_1^\epsilon}{2}
\right)
\right]
\geq0,
\end{equation}
where $\ETP_1^\epsilon$ and $\ET_1^\epsilon$ are associated to the Jacobi field matrix along
\begin{equation}
\geo_s^{\epsilon}=\exp_p(s\epsilon u),
\qquad
\Jf_0^\epsilon=\Id,
\qquad
\dot{\Jf}_0^\epsilon=A.
\end{equation}
We now apply Lemma \ref{lem:small_speed_entropy_tensor_expansion}. In the notation of that lemma take
\begin{equation}
a_n=\lambda,
\qquad
a_i=0
\quad\textnormal{for }i=1,\ldots,n-1.
\end{equation}
Since $\kappa_1=0$, we get
\begin{equation}
\label{eq:local_contrad}
\Tr[A-\ETP_1^\epsilon]
-
4\Tr\left[
\sinh^2\left(
\frac{\ET_1^\epsilon}{2}
\right)
\right]
=
\epsilon^2
\left[
\kappa_n
\frac{\lambda^2+6\lambda+6}{6(1+\lambda)}
+
\sum_{i=2}^{n-1}\kappa_i
\right]
+
O(\epsilon^3).
\end{equation}
Since
\begin{equation}
\frac{\lambda^2+6\lambda+6}{6(1+\lambda)}
\sim
\frac{\lambda}{6}
\qquad
\textnormal{as }\lambda\to\infty,
\end{equation}
and $\kappa_n<0$, we may choose $\lambda>0$ large enough so that
\begin{equation}
\kappa_n
\frac{\lambda^2+6\lambda+6}{6(1+\lambda)}
+
\sum_{i=2}^{n-1}\kappa_i
<0.
\end{equation}
For this choice of $\lambda$, the expansion \eqref{eq:local_contrad} is negative for all sufficiently small $\epsilon>0$, contradicting the pointwise inequality \eqref{eq:pointwiseineqforWasserContraction}. Hence, no negative sectional curvature direction exists.
\end{proof}
\bibliographystyle{amsplain}
\bibliography{ref}
\end{document}